\documentclass[a4paper,reqno]{amsart}
\usepackage{indentfirst}
\usepackage{amsfonts}
\usepackage{stmaryrd}
\usepackage{amsthm}
\usepackage{amssymb}
\usepackage{amsmath}
\usepackage{thmtools}
\usepackage[disable]{todonotes}
\usetikzlibrary{decorations.markings}
\usepackage{enumitem}
\usepackage{setspace}
\usepackage{hyperref}
\usetikzlibrary{matrix,arrows,positioning}

\makeatletter
\providecommand\@dotsep{5}
\renewcommand{\listoftodos}[1][\@todonotes@todolistname]{%
  \@starttoc{tdo}{#1}}
\makeatother

\newtheorem{Lem}{Lemma}[section]
\newtheorem{Prop}[Lem]{Proposition}
\newtheorem*{Def}{Definition}

\theoremstyle{plain}
\newtheorem{Thm}[Lem]{Theorem}
\newtheorem{Cor}[Lem]{Corollary}

\theoremstyle{definition}
\declaretheorem[numbered=no,name=Example,qed={\lower-0.3ex\hbox{$\triangleleft$}}]{Ex}
\newtheorem*{Rem}{Remark}

{\unskip\nobreak\hskip 1em plus 1fil\nobreak$\Box$
\parfillskip=0pt%
\endtrivlist}

\newcommand{\Dim}{\text{\textnormal{\textbf{dim}}}}
\newcommand{\Sym}{\text{\textnormal{Sym}}}
\newcommand{\Hom}{\text{\textnormal{Hom}}}

\newcommand{\Ext}{\text{\textnormal{Ext}}}
\newcommand{\codim}{\text{\textnormal{codim}}}
\newcommand{\git}{/\!\!/}

\mathchardef\mhyphen="2D

\def\itemNum$#1${\item $\displaystyle#1$
   \hfill\refstepcounter{equation}(\theequation)}
\def\pser#1{[\![#1]\!]} %formal power series [[#1]]
\begin{document}

\title[Orientifold DT invariants]{Cohomological orientifold Donaldson-Thomas invariants as Chow groups}

\author[H. Franzen]{Hans Franzen}
\address{Mathematisches Institut der Universit\"{a}t Bonn \\
Endenicher Allee 60, 53115 Bonn}
\email{franzen@math.uni-bonn.de}

\author[M.\,B. Young]{Matthew B. Young}
\address{Department of Mathematics\\
The University of Hong Kong\\
Pokfulam, Hong Kong}
\email{myoung@maths.hku.hk}

\date{\today}

%\subjclass[2010]{Primary: XXXXX ; Secondary XXXXX}

\begin{abstract}
We establish a geometric interpretation of orientifold Donaldson-Thomas invariants of $\sigma$-symmetric quivers with involution. More precisely, we prove that the cohomological orientifold Donaldson-Thomas invariant is isomorphic to the rational Chow group of the moduli space of $\sigma$-stable self-dual quiver representations. As an application we prove that the Chow Betti numbers of moduli spaces of stable $m$-tuples in classical Lie algebras can be computed numerically. We also prove a cohomological wall-crossing formula relating semistable Hall modules for different stabilities.
\end{abstract}

\maketitle

\tableofcontents

\setcounter{footnote}{0}

\section*{Introduction}
\addtocontents{toc}{\protect\setcounter{tocdepth}{1}}

There are by now a number of geometric interpretations of Donaldson-Thomas invariants of quivers. In various settings and levels of generality, see \cite{hausel2013}, \cite{chen2014}, \cite{meinhardt2014}, \cite{franzen2015b}, \cite{letellier2015}, \cite{davison2016a}. In particular, in \cite{franzen2015b} it is proved that the cohomological Donaldson-Thomas invariant of a symmetric quiver is isomorphic to the rational Chow group of the moduli space of stable representations of the quiver. In this paper we establish an analogous result for orientifold Donaldson-Thomas invariants, giving the first geometric interpretation of these invariants. Roughly speaking, orientifold Donaldson-Thomas invariants are counting invariants of moduli spaces of quiver representations which have orthogonal or symplectic structure groups.

The slope $\mu$ motivic Donaldson-Thomas invariant of a quiver $Q$ with stability $\theta$ is defined in terms of the Hilbert-Poincar\'{e} series of the semistable cohomological Hall algebra $\mathcal{H}_{Q,\mu}^{\theta \mhyphen ss}$, as defined by Kontsevich and Soibelman \cite{kontsevich2011}. The algebra $\mathcal{H}_{Q,\mu}^{\theta \mhyphen ss}$ is a convolution algebra on the cohomology of the stack of semistable representations of slope $\mu$. Moreover, a theory of cohomologically refined Donaldson-Thomas invariants is most naturally defined in terms of the algebra $\mathcal{H}_Q^{\theta \mhyphen ss}$ itself. When the quiver has an involution and duality structure (see Section \ref{sec:sdReps}) there is an associated semistable cohomological Hall module $\mathcal{M}_Q^{\theta \mhyphen ss}$, a left $\mathcal{H}_{Q,\mu=0}^{\theta \mhyphen ss}$-module structure on the cohomology of the stack of semistable self-dual representations, in terms of which cohomological orientifold Donaldson-Thomas theory is defined \cite{mbyoung2016b}.

Motivated by results of \cite{franzen2015b}, in this paper we study a Chow theoretic version of $\mathcal{M}_Q^{\theta \mhyphen ss}$, defined in the same way as $\mathcal{M}_Q^{\theta \mhyphen ss}$ but with Chow groups used in place of cohomology groups. The resulting object $\mathcal{B}_Q^{\theta \mhyphen ss}$ is a left module over the Chow theoretic Hall algebra $\mathcal{A}_{Q,\mu =0}^{\theta \mhyphen ss}$ introduced in \cite{franzen2016}. While cohomological and Chow theoretic Hall algebras share many common features, one technical advantage of the latter is the existence of the localization exact sequence in Chow theory. Using the localization exact sequence, the problem of understanding the trivial stability Hall algebra $\mathcal{A}_Q$ can be reduced to that of understanding the Chow groups of the Harder-Narasimhan strata of the stack of representations and the closely related algebras $\mathcal{A}_{Q,\mu}^{\theta \mhyphen ss}$. The Harder-Narasimhan stratification has been used extensively in the representation theory of quivers (see \cite{reineke2003} for its introduction) and is central to the approach of \cite{franzen2015b}. In its original setting, the Harder-Narasimhan stratification was used by Atiyah and Bott \cite{atiyah1983} in their study of two dimensional Yang-Mills theory with compact gauge group $\mathsf{H}$. See also \cite{laumon1996} for related work. The case in which $\mathsf{H}$ is a unitary group is analogous to the case of ordinary representations of quivers, but the framework is valid for arbitrary $\mathsf{H}$. The approach of this paper is to import the method of Atiyah and Bott when $\mathsf{H}$ is an orthogonal or symplectic group. The result is the $\sigma$-Harder-Narasimhan stratification of the stack of self-dual quiver representations. The main geometric properties of this stratification are summarized in Propositions \ref{prop:sdHNStrat} and \ref{prop:sdHNStratClose}, both of which are direct analogues of results in the ordinary case. The key result we prove using the $\sigma$-Harder-Narasimhan stratification is Proposition \ref{prop:noOddCohom}, which states that the mixed Hodge structure on the cohomology of the stack of semistable self-dual representations is pure of Hodge-Tate type. In the ordinary setting there are a number of proofs of the analogous purity statement (see \cite{franzen2015b}, \cite{davison2016a}), all of which use in some way framed quiver representations. Because of technical problems associated with framing in the orientifold setting, we instead prove Proposition \ref{prop:noOddCohom} by first showing that the $\sigma$-Harder-Narasimhan stratification is equivariantly perfect in the sense of \cite{atiyah1983}. Using Proposition \ref{prop:noOddCohom} we show in Theorem \ref{thm:moduleComparison} that the cycle map defines an isomorphism $\mathcal{B}_Q^{\theta \mhyphen ss} \xrightarrow[]{\sim} \mathcal{M}_Q^{\theta \mhyphen ss}$ which is compatible with the corresponding algebra isomorphism $\mathcal{A}_{Q,\mu=0}^{\theta \mhyphen ss} \xrightarrow[]{\sim} \mathcal{H}_{Q,\mu=0}^{\theta \mhyphen ss}$ from \cite{franzen2015b}. We deduce from this result that the cycle map from the Chow group to the cohomology of the moduli scheme of $\sigma$-stable self-dual representations surjects onto the pure part; see Corollary \ref{cor:purePart} and  Theorem \ref{thm:purePartOrdinary} for a strengthening of this result in the ordinary setting. In Theorem \ref{thm:hnFactor} we prove a wall-crossing formula for Chow theoretic Hall modules. More precisely, we show that as a graded vector space the trivial stability Hall module $\mathcal{B}_Q$ is determined by the semistable module $\mathcal{B}_Q^{\theta \mhyphen ss}$ and algebras $\{ \mathcal{A}_{Q,\mu}^{\theta \mhyphen ss} \}_{\mu > 0}$. Passing to Grothendieck groups, Theorem \ref{thm:hnFactor} recovers the motivic orientifold wall-crossing formula of \cite{mbyoung2015}.

In Section \ref{sec:oriDT} we give some applications of the above results to Donaldson-Thomas theory. We prove in Theorem \ref{thm:sdChowInterpretation} that the cohomological orientifold Donaldson-Thomas invariant of a $\sigma$-symmetric quiver is isomorphic to the rational Chow group of the moduli scheme of $\sigma$-stable self-dual representations. An immediate corollary is the confirmation of the orientifold variant of the Kontsevich-Soibelman integrality conjecture for $\sigma$-symmetric quivers; see Corollary \ref{cor:intConj}. In Theorem \ref{thm:noWallCrossing} we prove that the orientifold Donaldson-Thomas invariants of $\sigma$-symmetric quivers, and hence the Chow groups of their $\sigma$-stable moduli spaces, are independent of the choice of stability. In Section \ref{sec:loopQuiverApplication} we restrict attention to the quivers $L_m$ with one node and $m \geq 0$ loops. For particular choices of duality structures, moduli schemes of self-dual representations of $L_m$ are moduli schemes of $m$-tuples in orthogonal or symplectic Lie algebras. It follows from Theorem \ref{thm:sdChowInterpretation}, together with Reineke's computation of the Donaldson-Thomas invariants of $L_m$ \cite{reineke2012} and the freeness of $\mathcal{M}_{L_m}$ \cite{mbyoung2016b}, that the Chow Betti numbers of moduli schemes of $\sigma$-stable self-dual representations of $L_m$ can be computed numerically. As a final application, using Theorem \ref{thm:sdChowInterpretation} and the explicit shuffle description of $\mathcal{M}_{L_m}$ from \cite{mbyoung2016b} we show in Theorem \ref{thm:orthSymDuality} that the Chow Betti numbers of moduli schemes of stable $m$-tuples in symplectic and odd orthogonal Lie algebras groups agree.

While closely related, the present paper is mostly independent of \cite{mbyoung2016b}. Only in Section \ref{sec:loopQuiverApplication} do we use a difficult result from \cite{mbyoung2016b}.

\subsection*{Notation}

Throughout the paper we write $\otimes$ for $\otimes_{\mathbb{Q}}$, unless otherwise noted. All varieties/schemes are over the ground field $\mathbb{C}$. By the dimension of a smooth variety we always mean its complex dimension.

\subsection*{Acknowledgements}
The authors would like to thank Ben Davison for helpful correspondences. This work grew out of a visit of M.Y. to the Hausdorff Center for Mathematics. He would like to thank Tobias Dyckerhoff for the invitation and support. H.F.\ was supported by the DFG SFB / Transregio 45 ``Perioden, Modulr\"aume und Arithmetik algebraischer Variet\"aten''. M.Y. was partially supported by the Research Grants Council of the Hong Kong SAR, China (GRF HKU 703712) and by HKU URC Conference Support for Teaching Staff (award no. 143414).

\section{Background material}

\subsection{Quiver representations}
\addtocontents{toc}{\protect\setcounter{tocdepth}{2}}

Let $Q=(Q_0, Q_1)$ be a finite quiver and let $\Lambda_Q^+ =  \mathbb{Z}_{\geq 0} Q_0$ be its monoid of dimension vectors. A representation $(U,u)$ of $Q$ is a finite dimensional $Q_0$-graded complex vector space $U = \bigoplus_{i\in Q_0} U_i$ together with a linear map $u_{\alpha}: U_i \rightarrow U_j$ for each arrow $i \xrightarrow[]{\alpha} j$. The dimension vector $\Dim\, U \in \Lambda_Q^+$ of $U$ is the tuple $(\dim_{\mathbb{C}} U_i)_{i \in Q_0}$ and the total dimension $\dim U \in \mathbb{Z}_{\geq 0}$ is the dimension of $U$ as a complex vector space. We also define the total dimension $\dim d$ of a dimension vector as $\sum_{i \in Q_0} d_i$.

The affine variety of representations of $Q$ of dimension vector $d \in \Lambda_Q^+$ is
\[
R_d = \bigoplus_{i \xrightarrow[]{\alpha} j } \Hom_{\mathbb{C}}(\mathbb{C}^{d_i}, \mathbb{C}^{d_j}).
\]
The group $\mathsf{GL}_d =\prod_{i \in Q_0} \mathsf{GL}_{d_i}(\mathbb{C})$ acts linearly on $R_d$ via change of basis. The stack of representations of dimension vector $d$ is $\mathbf{M}_d = [R_d \slash \mathsf{GL}_d]$.

Write $\mathsf{Rep}_{\mathbb{C}}(Q)$ for the hereditary abelian category of finite dimensional complex representations of $Q$. The Euler form of $\mathsf{Rep}_{\mathbb{C}}(Q)$ descends to the bilinear form on $\Lambda_Q=\mathbb{Z} Q_0$ given by
\[
\chi(d, d^{\prime}) = \sum_{i \in Q_0} d_i d_i^{\prime} - \sum_{i \xrightarrow[]{\alpha} j } d_i d_j^{\prime}.
\]

\subsection{Self-dual quiver representations}

\label{sec:sdReps}

For an introduction to self-dual quiver representations see \cite{derksen2002}, or from the point of view of this paper, \cite{mbyoung2015}.

An involution of a quiver $Q$ is a pair of involutions $\sigma : Q_k \rightarrow Q_k$, $k=0,1$, such that
\begin{enumerate}[label=(\roman*)]
\item if $i \xrightarrow[]{\alpha} j$, then $\sigma(j) \xrightarrow[]{\sigma(\alpha)} \sigma(i)$, and
\item if $i \xrightarrow{\alpha} \sigma(i)$, then $\alpha = \sigma(\alpha)$.
\end{enumerate}
A duality structure on $(Q,\sigma)$ is a pair of functions $s: Q_0 \rightarrow \{ \pm 1 \}$ and $\tau: Q_1 \rightarrow \{ \pm 1 \}$ such that $s$ is $\sigma$-invariant and $\tau_{\alpha} \tau_{\sigma(\alpha)} = s_i s_j$ for each arrow $i \xrightarrow[]{\alpha} j$.

Fix an involution and duality structure on $Q$. A self-dual representation of $Q$ is a representation $(M,m)$ with a nondegenerate bilinear form $\langle \cdot, \cdot \rangle$ having the following properties:
\begin{enumerate}[label=(\roman*)]
\item $M_i$ and $M_j$ are orthogonal unless $i =\sigma(j)$,

\item the restriction of $\langle \cdot, \cdot \rangle$ to $M_i + M_{\sigma(i)}$ satisfies $\langle x,x^{\prime} \rangle = s_i \langle x^{\prime}, x \rangle$, and

\item for each arrow $i \xrightarrow[]{\alpha} j$ the structure maps of $M$ satisfy
\begin{equation}
\label{eq:strSymm}
\langle m_{\alpha} x,x^{\prime} \rangle - \tau_{\alpha}  \langle x, m_{\sigma(\alpha)} x^{\prime} \rangle =0, \qquad x \in M_i, \; x^{\prime} \in M_{\sigma(j)}.
\end{equation}
\end{enumerate}
The dimension vector of a self-dual representation lies in the submonoid $\Lambda_Q^{\sigma,+} \subset \Lambda_Q^+$ of dimension vectors which are fixed by $\sigma$.

Categorically, a choice of duality structure induces an exact contravariant functor $S: \mathsf{Rep}_{\mathbb{C}}(Q) \rightarrow \mathsf{Rep}_{\mathbb{C}}(Q)$ and an isomorphism of functors $\Theta: \mathbf{1}_{\mathsf{Rep}(Q)} \xrightarrow[]{\sim} S^2$ which satisfies $S(\Theta_U) \Theta_{S(U)} = \mathbf{1}_{S(U)}$, thus giving $\mathsf{Rep}_{\mathbb{C}}(Q)$ the structure of an abelian category with duality. In this language, a self-dual representation is a pair consisting of a representation $M$ and an isomorphism $M \xrightarrow[]{\sim} S(M)$ which is symmetric with respect to $\Theta_M$.

Let $e \in \Lambda_Q^{\sigma,+}$. Throughout the paper we will assume that $e_i$ is even whenever $i \in Q_0^{\sigma}$ and $s_i=-1$. The trivial representation $\mathbb{C}^e = \bigoplus_{i\in Q_0} \mathbb{C}^{e_i}$ then admits a self-dual structure which is unique up to isometry. The affine variety of self-dual representations of dimension vector $e$ can be identified with the subspace of $R_e$ consisting of structure maps which satisfy equation \eqref{eq:strSymm}. Explicitly, fixing partitions $Q_0 = Q_0^- \sqcup Q_0^{\sigma} \sqcup Q_0^+$ and $Q_1 = Q_1^- \sqcup Q_1^{\sigma} \sqcup Q_1^+$ such that $Q_0^{\sigma}$ consists of the nodes fixed by $\sigma$ and $\sigma(Q_0^-) = Q_0^+$, and similarly for $Q_1$, we have
\[
R_e^{\sigma} \simeq \bigoplus_{ i \xrightarrow[]{\alpha} j  \in Q_1^+} \Hom_{\mathbb{C}} (\mathbb{C}^{e_i}, \mathbb{C}^{e_j} ) \oplus \bigoplus_{ i \xrightarrow[]{\alpha} \sigma(i)  \in Q_1^{\sigma}} \mbox{Bil}^{s_i \tau_{\alpha}}(\mathbb{C}^{e_i})
\]
where $\mbox{Bil}^{\pm 1}(\mathbb{C}^{e_i})$ denotes the vector space of symmetric ($+$) or skew-symmetric ($-$) bilinear forms on $\mathbb{C}^{e_i}$. The isometry group of $\mathbb{C}^e$ is
\[
\mathsf{G}_e^{\sigma} \simeq  \prod_{i \in Q_0^+} \mathsf{GL}_{e_i} (\mathbb{C}) \times \prod_{i \in Q_0^{\sigma}} \mathsf{G}_{e_i}^{s_i}
\]
where
\[
\mathsf{G}_{e_i}^{s_i} = \left\{ \begin{array}{ll} \mathsf{Sp}_{e_i}(\mathbb{C}), & \mbox{ if } s_i=-1 \\ \mathsf{O}_{e_i}(\mathbb{C}), & \mbox{ if } s_i=1. \end{array} \right.
\]
The stack of self-dual representations of dimension vector $e$ is $\mathbf{M}_e^{\sigma} = [R_e^{\sigma} \slash \mathsf{G}_e^{\sigma}]$.

From the categorical point of view it is clear that for each representation $U$ the vector space $\Ext^i(S(U),U)$ is naturally a $\mathbb{Z}_2$-representation. By $\Ext^i(S(U),U)^{\pm S}$ we denote the subspace of symmetric ($+$) or skew-symmetric ($-$) elements. The analogue of the Euler form in the self-dual setting is
\[
\mathcal{E}(U) =\dim_{\mathbb{C}} \, \Hom(S(U),U)^{-S} - \dim_{\mathbb{C}} \, \Ext^1(S(U),U)^S.
\]
This descends to a function $\mathcal{E} : \Lambda_Q \rightarrow \mathbb{Z}$ given by
\begin{align*}
\mathcal{E}(d) =\sum_{i \in Q_0^{\sigma}} \frac{d_i(d_i -s_i)}{2}  + &   \sum_{i \in Q_0^+}  d_{\sigma(i)} d_i  - \\
  & \sum_{\sigma(i) \xrightarrow[]{\alpha} i \in Q_1^{\sigma}} \frac{d_i(d_i + \tau_{\alpha} s_i)}{2} -\sum_{i \xrightarrow[]{\alpha} j  \in Q_1^+} d_{\sigma(i)} d_j.
\end{align*}

\subsection{Moduli spaces of quiver representations}

We recall the construction of moduli spaces of quiver representations using geometric invariant theory (GIT) \cite{king1994}. Fix an element $\theta \in \Hom_{\mathbb{Z}}(\Lambda_Q, \mathbb{Z})$, called a stability. A representation $V$ of $Q$ is called semistable if $\mu(U) \leq \mu(V)$ for all non-trivial subrepresentations $U \subset V$ and is called stable if this inequality is strict. Here $\mu(V) = \frac{\theta(\Dim\, V)}{\dim V} \in \mathbb{Q}$ is the slope of $V$. There are $\mathsf{GL}_d$-invariant open (possibly empty) subvarieties $R_d^{\theta \mhyphen st} \subset R_d^{\theta \mhyphen ss} \subset R_d$ of (semi)stable representations. The moduli scheme $\mathfrak{M}^{\theta \mhyphen ss}_d$ of semistable representations is the GIT quotient $R_d \git_{\theta} \mathsf{GL}_d$, the stability $\theta$ determining the linearization of the group action. The stack $\mathbf{M}_e^{\theta \mhyphen st} = [R_d^{\theta \mhyphen st} \slash \mathsf{GL}_d]$ is a $\mathbb{C}^{\times}$-gerbe over the smooth moduli scheme $\mathfrak{M}^{\theta \mhyphen st}_d$ of stable representations, which is an open subvariety of $\mathfrak{M}^{\theta \mhyphen ss}_d$.

Suppose now that $Q$ has an involution and a duality structure. In this case we will always assume that $\theta$ is $\sigma$-compatible in the sense that $\sigma^*\theta = - \theta$. A self-dual representation $M$ is called $\sigma$-semistable if $\mu(U) \leq \mu(M)$ for all isotropic subrepresentations $U \subset M$  and is called $\sigma$-stable if this inequality is strict. Note that the slope of a self-dual representation is necessarily zero. The representation theoretic notion of $\sigma$-(semi)stability agrees with the corresponding notion in GIT \cite[Theorem 3.7]{mbyoung2015}. Hence there are $\mathsf{G}_e^{\sigma}$-invariant open subschemes $R_e^{\sigma, \theta \mhyphen st} \subset R_e^{\sigma, \theta \mhyphen ss} \subset R_e^{\sigma}$ of $\sigma$-(semi)stable self-dual representations. The stack $\mathbf{M}_e^{\sigma, \theta \mhyphen st}= [R_e^{\sigma, \theta \mhyphen st} \slash \mathsf{G}_e^{\sigma}]$ is a smooth Deligne-Mumford stack. The $\sigma$-stable moduli scheme $\mathfrak{M}_e^{\sigma, \theta \mhyphen st}$ is thus an open subscheme of $\mathfrak{M}_e^{\sigma, \theta\mhyphen ss} = R_e^{\sigma} \git_{\theta} \mathsf{G}_e^{\sigma}$ with at worst finite quotient singularities. By convention we set $\mathfrak{M}_0^{\sigma, \theta \mhyphen ss} = \mathfrak{M}_0^{\sigma, \theta \mhyphen st} = \mathsf{Spec}(\mathbb{C})$.

\section{Harder-Narasimhan stratifications}
\label{sec:hnStratifications}

\subsection{The $\sigma$-HN stratification}

Let $Q$ be an arbitrary quiver with stability $\theta$. Recall that each representation $U$ admits a unique Harder-Narasimhan (HN) filtration \cite[Proposition 2.5]{reineke2003}. This is an increasing filtration $0 = U_0 \subset U_1 \subset \cdots \subset U_r = U$ with the property that the subquotients $U_1 \slash U_0, \dots, U_r \slash U_{r-1}$ are semistable and satisfy
\[
\mu(U_1 \slash U_0) > \mu(U_2 \slash U_1) > \cdots > \mu(U_r \slash U_{r-1}).
\]

Suppose that $Q$ has an involution and a duality structure and let $M$ be a self-dual representation. If $U \subset M$ is an isotropic subrepresentation, then the orthogonal $U^{\perp} \subset M$ contains $U$ as a subrepresentation and the quotient $U^{\perp} \slash U$ inherits from $M$ a canonical self-dual structure. Denote by $M \git U$ the self-dual representation $U^{\perp} \slash U$.

\begin{Prop}[{\cite[\S 3.1]{mbyoung2015}}]
\label{prop:sdStability}
Let $\theta$ be a $\sigma$-compatible stability.
\begin{enumerate}
\item A representation $U$ is semistable of slope $\mu$ if and only if $S(U)$ is semistable of slope $-\mu$.

\item A self-dual representation is $\sigma$-semistable if and only if it is semistable as an ordinary representation.

\item Each self-dual representation $M$ admits a unique $\sigma$-HN filtration, that is, an isotropic filtration
\begin{equation}
\label{eq:sdHNFiltr}
0 = U_0 \subset U_1 \subset \cdots \subset U_r \subset M
\end{equation}
such that the subquotients $U_1 \slash U_0, \dots, U_r \slash U_{r-1}$ are semistable and satisfy
\[
\mu(U_1 \slash U_0) > \mu(U_2 \slash U_1) > \cdots > \mu(U_r \slash U_{r-1}) > 0
\]
and, if non-zero, $M \git U_r$ is $\sigma$-semistable.
\end{enumerate}
\end{Prop}

Using Proposition \ref{prop:sdStability} it is straightforward to show that if the $\sigma$-HN filtration of $M$ is given by \eqref{eq:sdHNFiltr}, then
\begin{equation}
\label{eq:longHNFilt}
0 = U_0 \subset U_1 \subset \cdots \subset U_r \subset U_r^{\perp} \subset \cdots \subset U_0^{\perp} = M
\end{equation}
is the HN filtration of $M$ considered as an ordinary representation, after identifying $U_r$ and $U_r^{\perp}$ if $M \git U_r$ is zero. Observe that the subquotients of the extended filtration \eqref{eq:longHNFilt} satisfy the symmetry conditions
\begin{equation}
\label{eq:dualSubQuot}
S(U_i \slash U_{i-1}) \simeq U_{i-1}^{\perp} \slash U_i^{\perp}, \qquad i=1, \dots, r.
\end{equation}

\begin{Def}
Let $r \geq 0$. A tuple $(d^{\bullet}, e^{\infty}) \in (\Lambda_Q^+)^r \times \Lambda_Q^{\sigma,+}$ is called a $\sigma$-HN type if the following conditions hold:
\begin{enumerate}
\item each $d^1, \dots, d^r$ is non-zero,
\item the slopes satisfy $\mu(d^1) > \dots > \mu(d^r) > 0$, and
\item each of $R_{d^1}^{\theta \mhyphen ss}, \dots, R_{d^r}^{\theta \mhyphen ss}$ and $R_{e^{\infty}}^{\sigma, \theta \mhyphen ss}$ is non-empty.
\end{enumerate}
The weight of $(d^{\bullet},e^{\infty})$ is defined to be $\sum_{j=1}^r (d^j + \sigma(d^j)) + e^{\infty}$.
\end{Def}

Note that $e^{\infty}$ is allowed to be zero in the above definition. We will write $\mathsf{HN}^{\sigma}(e)$ for the set of $\sigma$-HN types of weight $e \in \Lambda_Q^{\sigma,+}$. For $(d^{\bullet}, e^{\infty}) \in \mathsf{HN}^{\sigma}(e)$ denote by $R_{d^{\bullet},  e^{\infty}}^{\sigma, HN} \subset R^{\sigma}_e$ the subset of self-dual representations whose $\sigma$-HN filtration is of type $(d^{\bullet}, e^{\infty})$. Similarly, let $R_e^{\sigma,(d^{\bullet}, e^{\infty})} \subset R^{\sigma}_e$ be the subset of self-dual representations which have an isotropic filtration of type $(d^{\bullet},  e^{\infty})$. In the latter case we do not require that the subquotients of this filtration be $\sigma$-semistable. Analogous subsets $R^{HN}_{d^{\bullet}}, R_d^{d^{\bullet}} \subset R_d$ are defined for an ordinary HN type $d^{\bullet} \in \mathsf{HN}(d)$; see \cite{reineke2003}.

We need variations of two results of Reineke.

\begin{Prop}[{\textit{cf.} \cite[Proposition 3.4]{reineke2003}}]
\label{prop:sdHNStrat}
For each $e \in \Lambda_Q^{\sigma,+}$ the collection $\{R_{d^{\bullet}, e^{\infty}}^{\sigma, HN}\}_{(d^{\bullet}, e^{\infty}) \in \mathsf{HN}^{\sigma}(e)}$ defines a stratification of $R_e^{\sigma}$ by locally closed $\mathsf{G}^{\sigma}_e$-invariant smooth subschemes.
\end{Prop}

\begin{proof}
The argument is nearly the same as \cite{reineke2003}; we give it here for completeness. Let $(d^{\bullet}, e^{\infty}) \in \mathsf{HN}^{\sigma}(e)$ and let $\mathbb{C}^e$ be the trivial self-dual representation of dimension vector $e$. Let
\[
0 = E_0 \subset E_1 \subset \cdots \subset E_r \subset \mathbb{C}^e
\]
be a $Q_0$-graded isotropic filtration whose subquotients have dimension vector type $(d^{\bullet} , e^{\infty})$. Let $R^{\sigma}_{d^{\bullet}, e^{\infty}} \subset R_e^{\sigma}$ be the closed subscheme of self-dual representations which preserve $E_{\bullet} \subset \mathbb{C}^e$. The natural map
\[
\pi^{\sigma}: R^{\sigma}_{d^{\bullet}, e^{\infty}} \rightarrow R_{d^1} \times \cdots \times R_{d^r} \times R_{e^{\infty}}^{\sigma}
\]
is a trivial vector bundle. The preimage
\[
R^{\sigma, \theta \mhyphen ss}_{d^{\bullet}, e^{\infty}} = (\pi^{\sigma})^{-1} (R^{\theta \mhyphen ss}_{d^1} \times \cdots \times R^{\theta \mhyphen ss}_{d^r} \times R_{e^{\infty}}^{\sigma, \theta \mhyphen ss})
\]
is open in $R^{\sigma}_{d^{\bullet}, e^{\infty}}$. Let $\mathsf{G}_{d^{\bullet} , e^{\infty}}^{\sigma} \subset \mathsf{G}_e^{\sigma}$ be the parabolic subgroup which stabilizes $E_{\bullet} \subset \mathbb{C}^e$. Both $R^{\sigma}_{d^{\bullet}, e^{\infty}}$ and $R^{\sigma, \theta \mhyphen ss}_{d^{\bullet}, e^{\infty}}$ are $\mathsf{G}^{\sigma}_{d^{\bullet} , e^{\infty}}$-invariant subschemes of $R_e^{\sigma}$. Since the quotient $\mathsf{G}_e^{\sigma} \slash \mathsf{G}_{d^{\bullet} , e^{\infty}}^{\sigma}$ is projective, the action map
\[
m^{\sigma}: \mathsf{G}_e^{\sigma} \times_{\mathsf{G}_{d^{\bullet} , e^{\infty}}^{\sigma}} R^{\sigma}_{d^{\bullet}, e^{\infty}} \rightarrow R_e^{\sigma}
\]
is proper. It follows that the image of $m^{\sigma}$, which equals $R_e^{\sigma, (d^{\bullet}, e^{\infty})}$, is a closed subscheme of $R_e^{\sigma}$. The uniqueness of $\sigma$-HN filtrations implies that the restriction of $m^{\sigma}$ to $\mathsf{G}_e^{\sigma} \times_{\mathsf{G}_{d^{\bullet} , e^{\infty}}^{\sigma}} R^{\sigma, \theta \mhyphen ss}_{d^{\bullet}, e^{\infty}}$ defines an isomorphism
\begin{equation}
\label{eq:hnStrataIsom}
\mathsf{G}_e^{\sigma} \times_{\mathsf{G}_{d^{\bullet} , e^{\infty}}^{\sigma}} R^{\sigma, \theta \mhyphen ss}_{d^{\bullet}, e^{\infty}} \xrightarrow[]{\sim} R_{d^{\bullet} , e^{\infty}}^{\sigma, HN}.
\end{equation}
In particular, $R_{d^{\bullet} , e^{\infty}}^{\sigma, HN}$ is a smooth open subscheme of $R_e^{\sigma, (d^{\bullet}, e^{\infty})}$.
\end{proof}

Unlike the case of ordinary quiver representations, the $\sigma$-HN strata $R_{d^{\bullet} , e^{\infty}}^{\sigma, HN}$ need not be connected. This is because the isotropic flag variety $\mathsf{G}_e^{\sigma} \slash \mathsf{G}_{d^{\bullet} , e^{\infty}}^{\sigma}$ need not be irreducible. More precisely, $\mathsf{G}_e^{\sigma} \slash \mathsf{G}_{d^{\bullet} , e^{\infty}}^{\sigma}$ fails to be connected if and only if one of its factors is a flag variety for an orthogonal group $\mathsf{O}_{2n}(\mathbb{C})$ which parameterizes isotropic flags whose largest subspace is Lagrangian; such a flag variety has two irreducible components. The following example illustrates this behaviour.

\begin{Ex}
Let $Q$ be the quiver $\begin{tikzpicture}[thick,scale=.33,decoration={markings,mark=at position 0.6 with {\arrow{>}}}]
\draw[postaction={decorate}] (0,0) to  (2,0);
\draw[postaction={decorate}] (2,0) to  (4,0);
\fill (0,0) circle (4pt);
\fill (2,0) circle (4pt);
\fill (4,0) circle (4pt);
\end{tikzpicture}$ with its unique involution. Fix the duality structure $s=1$ and $\tau=-1$  (identically) and the stability $\theta=(1,0,-1)$. Let $e= (1,2,1)$ and consider the $\sigma$-HN type $((1,1,0), 0) \in \mathsf{HN}^{\sigma}(e)$. Then the $\sigma$-HN stratum $R_{(1,1,0), 0}^{\sigma, HN}$ consists of all self-dual representations of the form
\[
\mathbb{C} \xrightarrow[]{\left( \begin{smallmatrix} \lambda \\ 0 \end{smallmatrix} \right)} \mathbb{C}^2 \xrightarrow[]{\left( \begin{smallmatrix} 0 & - \lambda \end{smallmatrix} \right)} \mathbb{C} \qquad \mbox{ or } \qquad
\mathbb{C} \xrightarrow[]{\left( \begin{smallmatrix} 0 \\ \lambda \end{smallmatrix} \right)} \mathbb{C}^2 \xrightarrow[]{\left( \begin{smallmatrix} - \lambda & 0 \end{smallmatrix} \right)} \mathbb{C}
\]
for some $\lambda \in \mathbb{C}^{\times}$, where the vector space $\mathbb{C}^2$ attached to the middle node has ordered basis $\{x,y\}$ with symmetric bilinear form determined by
\[
\langle x,x\rangle = 0, \qquad \langle x, y \rangle =1, \qquad \langle y, y \rangle = 0.
\]
It follows that $R_{(1,1,0), 0}^{\sigma, HN}$ is a $\mathsf{O}_2(\mathbb{C})$-torsor and thus has two irreducible components.
\end{Ex}

For each $d^{\bullet} \in (\Lambda^+_Q)^r$ let $\mathcal{P}(d^{\bullet})$ be the polygon in $\mathbb{Z}_{\geq 0} \times \mathbb{Z}$ with vertices
\[
(\dim \sum_{i=1}^k  d^i, \theta (\sum_{i=1}^k d^i)), \qquad k =0, \dots, r.
\]
Following \cite{reineke2003}, define a partial order on tuples of $\Lambda_Q^+$ by $d^{\bullet} \leq d^{\prime \bullet}$ if $\mathcal{P}(d^{\bullet})$ lies on or below $\mathcal{P}(d^{\prime \bullet})$. Note that the polygon associated to a HN type is convex. Similarly, for a $\sigma$-HN type $(d^1, \dots, d^r, e^{\infty}) \in (\Lambda_Q^+)^r \times \Lambda_Q^{\sigma,+}$ we define $\mathcal{P}(d^\bullet,e^\infty)$ to be the polygon attached to the HN type $(d^1, \dots, d^r, e^{\infty}, \sigma(d^r), \dots, \sigma (d^1))$.
The polygon attached to a $\sigma$-HN type has a vertical reflection symmetry and lies in the subset $\mathbb{Z}_{\geq 0 } \times \mathbb{Z}_{\geq 0}$.

\begin{Prop}[{\textit{cf.} \cite[Proposition 3.7]{reineke2003}}]
\label{prop:sdHNStratClose}
Let $e \in \Lambda_Q^{\sigma,+}$. For each $(d^{\bullet}, e^{\infty}) \in \mathsf{HN}^{\sigma}(e)$ we have
\[
\overline{R_{d^{\bullet}, e^{\infty}}^{\sigma, HN}} \subset \bigcup_{\substack{(d^{ \prime \bullet}, e^{ \prime \infty}) \in \mathsf{HN}^{\sigma}(e) \\ (d^{\bullet}, e^{\infty}) \leq (d^{\prime \bullet} , e^{\prime \infty})}} R_{d^{ \prime \bullet}, e^{ \prime \infty}}^{\sigma, HN}.
\]
\end{Prop}

\begin{proof}
It is shown in the proof of \cite[Proposition 3.7]{reineke2003} that if $X$ is a representation with HN type $c^{\bullet}$ and $U \subset X$ is a subrepresentation, then the point $(\dim U, \theta(U))$ lies on or below $\mathcal{P}(c^{\bullet})$. In particular, this applies if $X$ is self-dual and $U$ is isotropic. We conclude that
\[
R_e^{\sigma,(d^{\bullet}, e^{\infty})} \subset \bigcup_{\substack{(d^{ \prime \bullet}, e^{ \prime \infty}) \in \mathsf{HN}^{\sigma}(e) \\ (d^{\bullet}, e^{\infty}) \leq (d^{\prime \bullet} , e^{\prime \infty})}} R_{d^{ \prime \bullet}, e^{ \prime \infty}}^{\sigma, HN}.
\]
But $\overline{R_{d^{\bullet}, e^{\infty}}^{\sigma, HN}} = R_e^{\sigma,(d^{\bullet}, e^{\infty})}$, as follows from the proof of Proposition \ref{prop:sdHNStrat}.
\end{proof}

\subsection{Equivariant perfection of the $\sigma$-HN stratification}

We show that the $\sigma$-HN filtration of $R_e^{\sigma}$ associated to a $\sigma$-compatible stability $\theta$ is $\mathsf{G}_e^{\sigma}$-equivariantly perfect in the sense of \cite{atiyah1983}, \cite[\S 2]{kirwan1984}. See \cite[\S 4]{harada2011}, \cite[\S 4]{brion2012b} for analogous results in the setting of ordinary quiver representations.

Let $X$ be a smooth connected complex algebraic variety with an action of a reductive algebraic group $\mathsf{G}$. Suppose that $\{S_i\}_{i=0}^N$ is a disjoint collection of locally closed $\mathsf{G}$-invariant smooth subschemes of $X$ such that $\bigsqcup_{i=0}^N S_i = X$ and
\[
\overline{S}_i \subset \bigcup_{ i \leq j} S_j
\]
for each $i =0, \dots, N$. Assume that each connected component of $S_i$ has complex codimension $\delta_i$ in $X$ and write $\nu_i \rightarrow S_i$ for the normal bundle of $S_i \subset X$. For each $i$ there is an associated long exact sequence of mixed Hodge structures, the equivariant Thom-Gysin sequence:
\begin{equation}
\label{eq:thomGysin}
\cdots \rightarrow H_{\mathsf{G}}^{k-2 \delta_i}(S_i)(-\delta_i) \rightarrow H_{\mathsf{G}}^k (\bigcup_{ i \leq j} S_j) \rightarrow H_{\mathsf{G}}^k(\bigcup_{i < j} S_j) \rightarrow H_{\mathsf{G}}^{k-2 \delta_i+1}(S_i)(-\delta_i) \rightarrow \cdots.
\end{equation}
Here $(-\delta_i) = - \otimes \mathbb{Q}(-1)^{\otimes \delta_i}$ with $\mathbb{Q}(-1)$ the Tate Hodge structure of weight $2$. It is shown in \cite{atiyah1983} that if each of the equivariant Euler classes $\mathsf{eu}_{\mathsf{G}}(\nu_i) \in H^{\bullet}_{\mathsf{G}}(S_i)$, $i=0, \dots, N$, is not a zero divisor, then the long exact sequence \eqref{eq:thomGysin} breaks into short exact sequences
\begin{equation}
\label{eq:thomShort}
0
\rightarrow H_{\mathsf{G}}^{k-2 \delta_i}(S_i)(-\delta_i) \rightarrow H_{\mathsf{G}}^k (\bigcup_{i \leq j} S_j) \rightarrow H_{\mathsf{G}}^k(\bigcup_{i < j} S_j) \rightarrow 0.
\end{equation}
In this situation the stratification $\{S_i\}_{i=0}^N$ is called $\mathsf{G}$-equivariantly perfect.

\begin{Lem}
\label{lem:generalOddVanishing}
Suppose that $\{S_i\}_{i=0}^N$ is a $\mathsf{G}$-equivariantly perfect stratification of $X$. If $H^{2k+1}_{\mathsf{G}}(X)=0$ and $H^{2k}_{\mathsf{G}}(X)$ is pure of Hodge-Tate type for all $k \in \mathbb{Z}$, then the same is true for $H^{\bullet}_{\mathsf{G}}(S_i)$, $i=0, \dots, N$.
\end{Lem}

\begin{proof}
This follows by induction on $i$ using the short exact sequences \eqref{eq:thomShort}.
\end{proof}

\begin{Prop}
\label{prop:noOddCohom}
Let $e \in \Lambda_Q^{\sigma,+}$ and $(d^{\bullet}, e^{\infty}) \in \mathsf{HN}^{\sigma}(e)$. Then $H^{2k+1}_{\mathsf{G}_e^{\sigma}}(R_{d^{\bullet}, e^{\infty}}^{\sigma, HN})=0$ and $H^{2k}_{\mathsf{G}_e^{\sigma}}(R_{d^{\bullet}, e^{\infty}}^{\sigma, HN})$ is pure of Hodge-Tate type for all $k \in \mathbb{Z}$.
\end{Prop}

\begin{proof}
We will apply Lemma \ref{lem:generalOddVanishing} with $X= R_e^{\sigma}$, $\mathsf{G} = \mathsf{G}_e^{\sigma}$ and the stratification $\{R_{d^{\bullet}, e^{\infty}}^{\sigma, HN} \}_{(d^{\bullet}, e^{\infty}) \in \mathsf{HN}^{\sigma}(e)}$ of Proposition \ref{prop:sdHNStrat} ordered as in Proposition \ref{prop:sdHNStratClose}. Note that $H^{\bullet}_{\mathsf{G}_e^{\sigma}} (R_e^{\sigma}) \simeq H^{\bullet}_{\mathsf{G}_e^{\sigma}}$ vanishes in odd degree and is of Hodge-Tate type in degree $2k$. Let $\nu_{d^{\bullet},e^{\infty}} \rightarrow R_{d^{\bullet},e^{\infty}}^{\sigma,HN}$ be the normal bundle of $R_{d^{\bullet},e^{\infty}}^{\sigma,HN} \subset R_e^{\sigma}$. We need to prove that $
\mathsf{eu}_{\mathsf{G}^{\sigma}_e}(\nu_{d^{\bullet}, e^{\infty}}) \in H^{\bullet}_{\mathsf{G}^{\sigma}_e}(R_{d^{\bullet},e^{\infty}}^{\sigma,HN})$ is not a zero divisor. The isomorphism \eqref{eq:hnStrataIsom} induces isomorphisms
\begin{eqnarray*}
H^{\bullet}_{\mathsf{G}^{\sigma}_e}(R_{d^{\bullet},e^{\infty}}^{\sigma, HN}) & \simeq & H_{\mathsf{G}^{\sigma}_{d^{\bullet},e^{\infty}}}^{\bullet}(R^{\sigma, \theta \mhyphen ss}_{d^{\bullet},e^{\infty}}) \\
& \simeq & H_{\mathsf{GL}_{d^1} \times \cdots \times \mathsf{GL}_{d^r} \times \mathsf{G}^{\sigma}_{e^{\infty}}}^{\bullet}(R^{\theta \mhyphen ss}_{d^1} \times \cdots \times R^{\theta \mhyphen ss}_{d^r} \times R_{e^{\infty}}^{\sigma, \theta \mhyphen ss}).
\end{eqnarray*}
Denote by $\tilde{\nu}_{d^{\bullet},e^{\infty}}$ the restriction of $\nu_{d^{\bullet},e^{\infty}}$ to $R^{\theta \mhyphen ss}_{d^1} \times \cdots \times R^{\theta \mhyphen ss}_{d^r} \times R_{e^{\infty}}^{\sigma, \theta \mhyphen ss}$. Then $\mathsf{eu}_{\mathsf{G}^{\sigma}_e}(\nu_{d^{\bullet}, e^{\infty}})$ is mapped to $\mathsf{eu}_{\mathsf{GL}_{d^1} \times \cdots \times \mathsf{GL}_{d^r} \times \mathsf{G}^{\sigma}_{e^{\infty}}}(\tilde{\nu}_{d^{\bullet},e^{\infty}})$ under the above isomorphisms. Fix $V_j \in R_{d^j}^{\theta \mhyphen ss}$, $j=1, \dots, r$, and $N \in R^{\sigma, \theta \mhyphen ss}_{e^{\infty}}$ and consider the orthogonal direct sum self-dual representation
\[
M = \bigoplus_{j=1}^r H(V_j) \oplus N \in R_e^{\sigma}.
\]
Here $H(V) = V \oplus S(V)$ is the hyperbolic self-dual representation on $V$. By using the isomorphisms \eqref{eq:dualSubQuot}, the fibre of $\tilde{\nu}_{d^{\bullet},e^{\infty}}$ over $M$ is naturally identified with the subspace of
\[
\bigoplus_{1 \leq i < j \leq r} \Ext^1(V_i, V_j) \oplus \bigoplus_{i=1}^r \Ext^1(V_i, N) \oplus \bigoplus_{1 \leq i \leq j \leq r} \Ext^1(V_i, S(V_j))
\]
consisting of elements whose component from $\oplus_{1 \leq i \leq j \leq r} \Ext^1(V_i, S(V_j))$ is fixed by the involution $S$.

Let $\mathsf{T} \simeq (\mathbb{C}^{\times})^r$ be the diagonal torus
\[
\mathsf{T} \subset \prod_{j=1}^r \mathsf{GL}_{d^j} \subset \prod_{j=1}^r \mathsf{GL}_{d^j} \times \mathsf{G}_{e^{\infty}}^{\sigma}.
\]
Then $\mathsf{T}$ acts trivially on $R_{d^1}^{\theta \mhyphen ss} \times \cdots \times R_{d^r}^{\theta \mhyphen ss}  \times R_{e^{\infty}}^{\sigma, \theta \mhyphen ss}$ and acts on $\tilde{\nu}_{d^{\bullet},e^{\infty} \vert_M}$ with weights
\begin{enumerate}
\item $t_j - t_i$ on the summand $\Ext^1(V_i, V_j)$,
\item $- t_i$ on the summand $\Ext^1(V_i, N)$, and
\item $-t_j- t_i$ on the summand $\Ext^1(V_i, S(V_j))$.
\end{enumerate}
We can therefore apply the Atiyah-Bott lemma \cite[Proposition 13.4]{atiyah1983} to conclude that $\mathsf{eu}_{\mathsf{G}^{\sigma}_e}(\nu_{d^{\bullet},e^{\infty}})$ is not a zero divisor.
\end{proof}

We describe here one application of Proposition \ref{prop:noOddCohom}. Let $\mathbb{F}_q$ be a finite field of odd order $q$. In \cite[Theorem 4.4]{mbyoung2015} the stacky number of $\mathbb{F}_q$-rational points $\# \mathbf{M}_e^{\sigma, \theta \mhyphen ss} (\mathbb{F}_q)$ was explicitly computed, the result being a rational function of $q$. A similar calculation shows that the class $[\mathbf{M}_e^{\sigma, \theta \mhyphen ss}] \in K_0(\mathsf{St}_{\mathbb{C}})$ in the Grothendieck ring of stacks over $\mathbb{C}$ is given by the same rational function, with $q$ replaced by the Lefschetz motive $\mathbb{L} = [\mathbb{A}^1_{\mathbb{C}}]$. The new ingredient in the motivic setting is the computation of the orthogonal factors of the motive of the classifying stack $B \mathsf{G}^{\sigma}_e$, which was done in \cite{dhillonyoung2016}. It follows from Proposition \ref{prop:noOddCohom} that the formula from \cite{mbyoung2015} in fact computes the Poincar\'{e} series of $\mathbf{M}_e^{\sigma, \theta \mhyphen ss}$ with variable $q^{\frac{1}{2}}$.

\section{Chow theoretic Hall algebras and modules}

\subsection{Basic definitions}

We recall the cohomological Hall algebra of Kontsevich and Soibelman along with some of its modifications.

Fix a quiver $Q$. Let $\mathsf{Vect}_{\mathbb{Z}}$ be the category of finite dimensional $\mathbb{Z}$-graded rational vector spaces. Write $D^{lb}(\mathsf{Vect}_{\mathbb{Z}}) \subset D(\mathsf{Vect}_{\mathbb{Z}})$ for the full subcategory of objects whose cohomological and $\mathbb{Z}$ degrees are bounded below. Let also $D^{lb}(\mathsf{Vect}_{\mathbb{Z}})_{\Lambda^+_Q}$ be the category whose objects are $\Lambda^+_Q$-graded objects of $D^{lb}(\mathsf{Vect}_{\mathbb{Z}})$ with finite dimensional $\Lambda_Q^+ \times \mathbb{Z}$-homogeneous summands and whose morphisms preserve the $\Lambda_Q^+ \times \mathbb{Z}$-grading. Denote by $\{ \frac{1}{2} \}$ the functor given by tensor product with the one dimensional vector space of cohomological and $\mathbb{Z}$ degree $-1$. Define a monoidal product $\boxtimes^{\mathsf{tw}}$ on $D^{lb}(\mathsf{Vect}_{\mathbb{Z}})_{\Lambda^+_Q}$ by
\[
\bigoplus_{d \in \Lambda^+_Q} \mathcal{U}^{\prime}_d  \boxtimes^{\mathsf{tw}}  \bigoplus_{d \in \Lambda^+_Q} \mathcal{U}^{\prime \prime}_{d}   = \bigoplus_{d \in \Lambda^+_Q} \Big( \bigoplus_{\substack{ (d^{\prime},d^{\prime \prime}) \in \Lambda_Q^+ \times \Lambda_Q^+ \\ d= d^{\prime} + d^{\prime \prime}}}  \mathcal{U}^{\prime}_{d^{\prime}} \otimes \mathcal{U}^{\prime \prime}_{d^{\prime \prime}}  \{(\chi(d^{\prime} , d^{\prime \prime}) - \chi(d^{\prime \prime} , d^{\prime}) )\slash 2 \} \Big).
\]

Let $\theta$ be a stability. For each slope $\mu \in \mathbb{Q}$ define
\[
\mathcal{H}^{\theta \mhyphen ss}_{Q,\mu} = \bigoplus_{d \in \Lambda_{Q, \mu}^+} H^{\bullet}_{\mathsf{GL}_d}(R^{\theta \mhyphen ss}_d) \{ \chi(d,d) \slash 2 \} \in D^{lb}(\mathsf{Vect}_{\mathbb{Z}})_{\Lambda^+_{Q,\mu}}
\]
where $\Lambda_{Q, \mu}^+ = \{ d \in \Lambda_Q^+ \mid \mu(d) = \mu\} \cup \{0\}$ and we use singular equivariant cohomology with rational coefficients. The $\mathbb{Z}$-grading of $\mathcal{H}^{\theta \mhyphen ss}_{Q,\mu}$ is the cohomological, or equivalently Hodge theoretic weight, grading. We interpret $H^{\bullet}_{\mathsf{GL}_d}(R^{\theta \mhyphen ss}_d)$ as the cohomology of the stack $\mathbf{M}_d^{\theta \mhyphen ss}=[R_d^{\theta \mhyphen ss} \slash \mathsf{GL}_d]$. For each $d, d^{\prime} \in \Lambda_{Q,\mu}^+$ there is a correspondence
\[
\begin{array}{ccccl}
\mathbf{M}^{\theta \mhyphen ss}_d \times \mathbf{M}^{\theta \mhyphen ss}_{d^{\prime}} & \xleftarrow[]{\pi_1 \times \pi_3} & \mathbf{M}^{\theta \mhyphen ss}_{d,d^{\prime}}  & \xrightarrow[]{\pi_2} & \mathbf{M}^{\theta \mhyphen ss}_{d+ d^{\prime}} \\ 
(U, V \slash U) & \longmapsfrom & U \subset V & \longmapsto & V
\end{array}
\]
with $\mathbf{M}^{\theta \mhyphen ss}_{d,d^{\prime}} = [R_{d,d^{\prime}}^{\theta \mhyphen ss} \slash \mathsf{GL}_{d,d^{\prime}}]$ the stack of flags of semistable representations of dimension vector type $(d,d^{\prime})$. The map $\pi_1 \times \pi_3$ is a homotopy equivalence while $\pi_2$ is proper. The composition $\pi_{2!} \circ (\pi_1 \times \pi_3)^*$ defines a product on $\mathcal{H}^{\theta \mhyphen ss}_{Q,\mu}$ making it into an associative algebra object in $D^{lb}(\mathsf{Vect}_{\mathbb{Z}})_{\Lambda^+_Q}$, called the slope $\mu$ semistable cohomological Hall algebra (CoHA) \cite{kontsevich2011}.

When $Q$ has a duality structure and $\theta$ is $\sigma$-compatible, the category $D^{lb}(\mathsf{Vect}_{\mathbb{Z}})_{\Lambda_Q^{\sigma,+}}$ becomes a left module over $(D^{lb}(\mathsf{Vect}_{\mathbb{Z}})_{\Lambda^+_Q}, \boxtimes^{\mathsf{tw}})$ via the formula
\[
\bigoplus_{d \in \Lambda^+_Q} \mathcal{U}_d \boxtimes^{S \mhyphen \mathsf{tw}} \bigoplus_{e \in \Lambda_Q^{\sigma,+}} \mathcal{V}_e = \bigoplus_{e \in \Lambda_Q^{\sigma,+}} \Big(\bigoplus_{\substack{ (d^{\prime},e^{\prime}) \in \Lambda_Q^+ \times \Lambda_Q^{\sigma,+} \\ e =H(d^{\prime}) + e^{\prime}}}  \mathcal{U}_{d^{\prime}} \otimes \mathcal{V}_{e^{\prime}}  \{ \gamma(d^{\prime}, e^{\prime}) \slash 2 \} \Big)
\]
where
\[
\gamma(d, e) = \chi( d, e ) - \chi(e , d) + \mathcal{E}(\sigma(d)) - \mathcal{E}(d).
\]
Define
\[
\mathcal{M}^{\theta \mhyphen ss}_Q = \bigoplus_{e \in \Lambda_Q^{\sigma,+}} H^{\bullet}_{\mathsf{G}^{\sigma}_e}(R^{\sigma,\theta \mhyphen ss}_e) \{ \mathcal{E}(e) \slash 2 \} \in D^{lb}(\mathsf{Vect}_{\mathbb{Z}})_{\Lambda_Q^{\sigma,+}}.
\]
If $d \in \Lambda_{Q, \mu=0}^+$ and $e \in \Lambda_Q^{\sigma,+}$, then there is a correspondence of stacks
\[
\begin{array}{ccccl}
\mathbf{M}^{\theta \mhyphen ss}_d \times \mathbf{M}^{\sigma, \theta \mhyphen ss}_{e} & \xleftarrow[]{\pi_1 \times \pi^{\sigma}_3} & \mathbf{M}^{\sigma, \theta \mhyphen ss}_{d,e}  & \xrightarrow[]{\pi^{\sigma}_2} & \mathbf{M}^{\sigma, \theta \mhyphen ss}_{d + \sigma(d) + e} \\ 
(U, M \git U) & \longmapsfrom & U \subset M & \longmapsto & M
\end{array}
\]
where $\mathbf{M}^{\sigma, \theta \mhyphen ss}_e = [R^{\sigma, \theta \mhyphen ss}_e \slash \mathsf{G}^{\sigma}_e]$ and $\mathbf{M}^{\sigma, \theta \mhyphen ss}_{d,e} = [R^{\sigma, \theta \mhyphen ss}_{d,e} \slash \mathsf{G}^{\sigma}_{d,e}]$. As above, $\pi_1 \times \pi_3^{\sigma}$ is a homotopy equivalence and $\pi_2^{\sigma}$ is proper. The composition $\pi^{\sigma}_{2!} \circ (\pi_1 \times \pi^{\sigma}_3)^*$ gives $\mathcal{M}^{\theta \mhyphen ss}_Q$ the structure of left $\mathcal{H}^{\theta \mhyphen ss}_{Q, \mu=0}$-module object in $D^{lb}(\mathsf{Vect}_{\mathbb{Z}})_{\Lambda_Q^{\sigma,+}}$, called the semistable cohomological Hall module (CoHM). See \cite{mbyoung2016b} for details. We denote by $\star$ the action of $\mathcal{H}^{\theta \mhyphen ss}_{Q, \mu=0}$ on $\mathcal{M}^{\theta \mhyphen ss}_Q$.

Let $\mathsf{W}(Q)$ be the abelian group $\mathbb{Z}_2 Q_0^{\sigma}$. We have an exact sequence of groups
\[
\Lambda_Q \xrightarrow[]{H} \Lambda_Q^{\sigma} \xrightarrow[]{\nu} \mathsf{W}(Q) \rightarrow 0
\]
where $H(d) = d + \sigma(d)$ and $\nu$ sends a dimension vector to its parities at $Q_0^{\sigma}$. Then there is a $\mathcal{H}_{Q, \mu=0}^{\theta \mhyphen ss}$-module decomposition $
\mathcal{M}_Q^{\theta \mhyphen ss} = \bigoplus_{w \in \mathsf{W}(Q)} \mathcal{M}_Q^{\theta \mhyphen ss}(w)$ with
\[
\mathcal{M}_Q^{\theta \mhyphen ss}(w) = \bigoplus_{\substack{e \in \Lambda_Q^{\sigma,+} \\ \nu(e) =w} } H^{\bullet}_{\mathsf{G}_e^{\sigma}}(R^{\sigma, \theta \mhyphen ss}_e) \{\mathcal{E}(e) \slash 2\}.
\]
Note that $\mathcal{M}_Q^{\theta \mhyphen ss}(w)$ is trivial unless $s_i=1$ whenever $w_i=1$.

In the case of trivial stability, $\theta =0$, we write $\mathcal{H}_Q$ and $\mathcal{M}_Q$ for the associated cohomological Hall algebra or module. There are explicit combinatorial descriptions of $\mathcal{H}_Q$ and $\mathcal{M}_Q$ in terms of shuffle algebras and signed shuffle modules; see \cite[Theorem 2]{kontsevich2011} and \cite[Theorem 3.3]{mbyoung2016b}, respectively.

There are also Chow theoretic versions of $\mathcal{H}^{\theta \mhyphen ss}_{Q,\mu}$ and $\mathcal{M}^{\theta \mhyphen ss}_Q$, called the Chow Hall algebra (ChowHA) and module (ChowHM) and denoted by $\mathcal{A}_{Q,\mu}^{\theta \mhyphen ss}$ and $\mathcal{B}_Q^{\theta \mhyphen ss}$. The former was introduced and studied in \cite{franzen2016}, \cite{franzen2015}, \cite{franzen2015b}. The definitions of $\mathcal{A}_{Q,\mu}^{\theta \mhyphen ss}$ and $\mathcal{B}_Q^{\theta \mhyphen ss}$ are entirely similar to those of their cohomological counterparts but with rational equivariant Chow groups used in place of rational equivariant cohomology groups. For example, the slope $\mu$ semistable ChowHA is
\[
\mathcal{A}^{\theta \mhyphen ss}_{Q,\mu} = \bigoplus_{d \in \Lambda_{Q, \mu}^+} A^{\bullet}_{\mathsf{GL}_d}(R^{\theta \mhyphen ss}_d)_{\mathbb{Q}} \{\chi(d,d) \slash 2 \} \in D^{lb}(\mathsf{Vect}_{\mathbb{Z}})_{\Lambda_{Q,\mu}^+}.
\]
The $\mathbb{Z}$-grading of $\mathcal{A}^{\theta \mhyphen ss}_{Q,\mu}$ is defined by putting the equivariant Chow group $A^k_{\mathsf{GL}_d}(R_d^{\theta \mhyphen ss})$ in degree $2k$. The semistable ChowHM is defined analogously. Like $\mathcal{M}_Q^{\theta \mhyphen ss}$, the $\mathcal{A}^{\theta \mhyphen ss}_{Q,\mu=0}$-module $\mathcal{B}_Q^{\theta \mhyphen ss}$ admits a direct sum decomposition labelled by $\mathsf{W}(Q)$. For background on equivariant Chow theory the reader is referred to \cite{edidin1998}, \cite{brion1998}.

\subsection{Comparison of the ChowHM and CoHM}

Let $\mathsf{G}$ be a linear algebraic group $\mathsf{G}$. We write $A^{\bullet}_{\mathsf{G}}$ for $A_{\mathsf{G}}^{\bullet}(\mbox{Spec} \, \mathbb{C})_{\mathbb{Q}}$ and similarly for $H^{\bullet}_{\mathsf{G}}$. For each $r \geq 1$ there is a graded ring isomorphism $A_{\mathsf{GL}_r}^{\bullet} \simeq \mathbb{Q}[x_1, \dots, x_r]$ with $x_i$ of degree $i$ and the equivariant cycle map $A_{\mathsf{GL}_r}^{\bullet} \rightarrow H_{\mathsf{GL}_r}^{\bullet}$ is a degree doubling ring isomorphism. In particular, $H_{\mathsf{GL}_r}^{\bullet}$ is concentrated in even degrees. Similarly, for $\mathsf{G}$ an orthogonal or symplectic group of rank $r \geq 1$ we have a ring isomorphism $A_{\mathsf{G}}^{\bullet} \simeq \mathbb{Q}[p_1, \dots, p_r ]$ with $p_i$ of degree $2i$ and the equivariant cycle map $A_{\mathsf{G}}^{\bullet} \rightarrow H_{\mathsf{G}}^{\bullet}$ is again a degree doubling isomorphism. Note that with integer coefficients the above statements are only valid for general linear and symplectic groups. Even with rational coefficients the statements are not true for the special orthogonal groups $\mathsf{SO}_{2r}$. See \cite{brown1982}, \cite{edidin1997},  \cite{pandharipande1998}, \cite{totaro1999} for proofs of these statements.

In \cite[Corollary 5.6]{franzen2015b} it is shown that for any stability $\theta$ and slope $\mu$ the equivariant cycle map $\mathsf{cl}_{alg}: \mathcal{A}_{Q,\mu}^{\theta \mhyphen ss} \rightarrow \mathcal{H}_{Q,\mu}^{\theta \mhyphen ss}$ is an isomorphism of algebra objects in $D^{lb}(\mathsf{Vect}_{\mathbb{Z}})_{\Lambda_{Q,\mu}^+}$. The next result lifts this isomorphism to Hall modules.

\begin{Thm}
\label{thm:moduleComparison}
Let $\theta$ be a $\sigma$-compatible stability. The equivariant cycle map $\mathsf{cl}_{mod} : \mathcal{B}^{\theta \mhyphen ss}_Q \rightarrow \mathcal{M}^{\theta \mhyphen ss}_Q$ is an isomorphism in $D^{lb}(\mathsf{Vect}_{\mathbb{Z}})_{\Lambda_Q^{\sigma,+}}$. Moreover, the diagram
\[
\begin{tikzpicture}
  \matrix (m) [matrix of math nodes,nodes={anchor=center},row sep=3em,column sep=3em,minimum width=1.4em] {
\mathcal{A}^{\theta \mhyphen ss}_{Q, \mu=0} \boxtimes^{S \mhyphen \mathsf{tw}} \mathcal{B}^{\theta \mhyphen ss}_Q & \mathcal{B}^{\theta \mhyphen ss}_Q  \\
\mathcal{H}^{\theta \mhyphen ss}_{Q, \mu=0} \boxtimes^{S \mhyphen \mathsf{tw}} \mathcal{H}^{\theta \mhyphen ss}_Q & \mathcal{M}^{\theta \mhyphen ss}_Q \\};
\draw  (m-1-1) edge[->] node[above]{$\star$} (m-1-2);
\draw  (m-1-1) edge[->] node[left]{$\mathsf{cl}_{alg} \boxtimes^{S \mhyphen \mathsf{tw}} \mathsf{cl}_{mod}$} (m-2-1);
\draw  (m-2-1) edge[->] node[below]{$\star$} (m-2-2);
\draw  (m-1-2) edge[->] node[right]{$\mathsf{cl}_{mod}$} (m-2-2);
\end{tikzpicture}
\]
commutes.
\end{Thm}

\begin{proof}
It suffices to prove the theorem when $\mathcal{B}^{\theta \mhyphen ss}_Q$ and $\mathcal{M}^{\theta \mhyphen ss}_Q$ are replaced with $\mathcal{B}^{\theta \mhyphen ss}_Q(w)$ and $\mathcal{M}^{\theta \mhyphen ss}_Q(w)$, respectively, for an arbitrary class $w \in \mathsf{W}(Q)$.

Proposition \ref{prop:noOddCohom} shows that $H^{2k+1}_{\mathsf{G}_e^{\sigma}}(R_e^{\sigma, \theta \mhyphen ss})$ vanishes for all $k \geq 0$, so for the first statement we need only prove that
\[
\mathsf{cl}_{mod}^e: A^k_{\mathsf{G}_e^{\sigma}}(R_e^{\sigma, \theta \mhyphen ss})_{\mathbb{Q}} \rightarrow H^{2k}_{\mathsf{G}_e^{\sigma}}(R_e^{\sigma, \theta \mhyphen ss})
\]
is a rational vector space isomorphism for each $k \geq 0$. Define a partial order on $\Lambda_Q^{\sigma,+}$ by $e^{\prime} \leq e$ if $e^{\prime}_i \leq e_i$ for all $i \in Q_0$. We proceed by induction on $e \in \Lambda_Q^{\sigma,+}$ with fixed class $w \in \mathsf{W}(Q)$. Let $e$ be the minimal such dimension vector. Explicitly, $e$ is zero except at those nodes $i \in Q_0^{\sigma}$ with $s_i=1$, in which case $e_i$ is zero if $w_i=0$ and is one if $w_i=1$. A self-dual representation of dimension vector $e$ has no isotropic subrepresentations. Hence $R_e^{\sigma, \theta \mhyphen ss} = R_e^{\sigma}$ and the cycle map reduces to $\mathsf{cl}_{mod}^e : A^k_{\mathsf{G}_e^{\sigma}} \rightarrow H^{2k}_{\mathsf{G}_e^{\sigma}}$, which is an isomorphism by the discussion at the start of this section.

Assume that $\mathsf{cl}_{mod}^{e^{\prime}}$ is an isomorphism for all $e^{\prime} < e$ of class $w$. Let $R_e^{\sigma, \theta \mhyphen unst} = R_e^{\sigma} - R_e^{\sigma, \theta \mhyphen ss}$ be the closed subscheme of unstable self-dual representations of dimension vector $e$. Using Proposition \ref{prop:sdHNStratClose} we obtain a stratification of $R_e^{\sigma, \theta \mhyphen unst}$ by $\mathsf{G}_e^{\sigma}$-invariant closed subschemes whose successive complements are of the form $R_{d^{\bullet}, e^{\infty}}^{\sigma, HN}$ for some $(d^{\bullet}, e^{\infty}) \in \mathsf{HN}^{\sigma}(e)$ different from $(e)$. Since $R_e^{\sigma, \theta \mhyphen ss}$ is smooth, there are isomorphisms
\[
A_{\bullet}^{\mathsf{G}_e^{\sigma}}(R_e^{\sigma, \theta \mhyphen ss}) \simeq A^{\dim R_e^{\sigma} - \bullet}_{\mathsf{G}_e^{\sigma}}(R_e^{\sigma, \theta \mhyphen ss})  ,  \qquad H_{\bullet}^{BM, \mathsf{G}_e^{\sigma}}(R_e^{\sigma, \theta \mhyphen ss}) \simeq H^{2 \dim R_e^{\sigma} - \bullet}_{\mathsf{G}_e^{\sigma}}(R_e^{\sigma, \theta \mhyphen ss}),
\]
where $H^{BM}_{\bullet}(-)$ denotes Borel-Moore homology with rational coefficients. Using the equivariant lift of \cite[Lemma 5.3]{franzen2015b}, to prove that $\mathsf{cl}_{mod}^e$ is an isomorphism it therefore suffices to prove that each of the cycle maps $A^{\bullet}_{\mathsf{G_e^{\sigma}}}(R_{d^{\bullet}, e^{\infty}}^{\sigma, HN})_{\mathbb{Q}} \rightarrow H^{\bullet}_{\mathsf{G_e^{\sigma}}}(R_{d^{\bullet}, e^{\infty}}^{\sigma, HN})$ is an isomorphism. Arguing as in the proof of Proposition \ref{prop:noOddCohom}, we find
\[
A_{\mathsf{G}_e^{\sigma}}^{\bullet}(R_{d^{\bullet}, e^{\infty}}^{\sigma, HN} ) \simeq A^{\bullet}_{\mathsf{GL}_{d^1} \times \cdots \times \mathsf{GL}_{d^r} \times \mathsf{G}_{e^{\infty}}^{\sigma}}(R_{d^1}^{\theta \mhyphen ss} \times \cdots \times R_{d^r}^{\theta \mhyphen ss} \times R_{e^{\infty}}^{\sigma,\theta \mhyphen ss}).
\]
Since $e^{\infty} < e$, the inductive hypothesis implies that $\mathsf{cl}_{mod}^{e^{\infty}}$ is an isomorphism while $\mathsf{cl}_{alg}^d$ is an isomorphism for all $d \in \Lambda_Q^+$ by \cite[Theorem 5.1]{franzen2015b}. As $R^{\sigma, \theta \mhyphen ss}_{e^{\infty}}$ has no odd degree $\mathsf{G}_{e^{\infty}}^{\sigma}$-equivariant cohomology, we can apply the equivariant version of \cite[Lemmas 6.1, 6.2]{totaro1999} (with rational coefficients) to conclude that the exterior product map
\begin{multline*}
A^{\bullet}_{\mathsf{GL}_{d^1}} (R_{d^1}^{\theta \mhyphen ss})_{\mathbb{Q}} \otimes \cdots \otimes A^{\bullet}_{\mathsf{GL}_{d^r}} (R_{d^r}^{\theta \mhyphen ss})_{\mathbb{Q}} \otimes   A^{\bullet}_{\mathsf{G}_{e^{\infty}}^{\sigma}}(R_{e^{\infty}}^{\sigma,\theta \mhyphen ss})_{\mathbb{Q}} \rightarrow \\
A^{\bullet}_{\mathsf{GL}_{d^1} \times \cdots \times \mathsf{GL}_{d^r} \times \mathsf{G}_{e^{\infty}}^{\sigma}}(R_{d^1}^{\theta \mhyphen ss} \times \cdots \times R_{d^r}^{\theta \mhyphen ss} \times R_{e^{\infty}}^{\sigma,\theta \mhyphen ss})_{\mathbb{Q}}
\end{multline*}
is an isomorphism. Similarly, the corresponding K\"{u}nneth map in equivariant cohomology is an isomorphism. Compatibility of the cycle map with exterior products then implies that $\mathsf{cl}_{mod}^e$ is an isomorphism.

That $\mathsf{cl}_{mod}$ respects gradings is clear. That $\mathsf{cl}_{mod}$ respects the Hall algebra actions follows from the fact that cycle maps are covariant for proper morphisms and contravariant for morphisms of smooth varieties.
\end{proof}

\begin{Rem}
Unlike $\mathsf{cl}_{alg}: \mathcal{A}_{Q,\mu}^{\theta \mhyphen ss} \rightarrow \mathcal{H}_{Q,\mu}^{\theta \mhyphen ss}$, the map $\mathsf{cl}_{mod} : \mathcal{B}^{\theta \mhyphen ss}_Q \rightarrow \mathcal{M}^{\theta \mhyphen ss}_Q$ is in general neither injective nor surjective if integer coefficients are used. However, the diagram from Theorem \ref{thm:moduleComparison} remains commutative over the integers.
\end{Rem}

\begin{Cor}
\label{cor:purePart}
Let $Q$ be a quiver with duality structure and $\sigma$-compatible stability $\theta$. For each $e \in \Lambda_Q^{\sigma,+}$ the cycle map $\mathsf{cl}: A^{\bullet}(\mathfrak{M}_e^{\sigma, \theta \mhyphen st})_{\mathbb{Q}} \rightarrow H^{\bullet}(\mathfrak{M}_e^{\sigma, \theta \mhyphen st})$ surjects onto the pure part
\[
PH^{\bullet}(\mathfrak{M}_e^{\sigma, \theta \mhyphen st}) = \bigoplus_{k \geq 0} W_k H^k(\mathfrak{M}_e^{\sigma, \theta \mhyphen st}).
\]
In particular, $PH^{\bullet}(\mathfrak{M}_e^{\sigma, \theta \mhyphen st})$ consists entirely of Hodge classes, that is,
\[
W_{2k} H^{2k}(\mathfrak{M}_e^{\sigma, \theta \mhyphen st}) = W_{2k}H^{2k}(\mathfrak{M}_e^{\sigma, \theta \mhyphen st}) \cap F^k H^{2k}(\mathfrak{M}_e^{\sigma, \theta \mhyphen st}; \mathbb{C})
\]
and $W_{2k+1} H^{2k+1}(\mathfrak{M}_e^{\sigma,\theta \mhyphen st})=0$ for all $k \geq 0$.
\end{Cor}

\begin{proof}
The method of proof of \cite[Theorem 1.1]{chen2014} can be used to show that the restriction map $H_{\mathsf{G}_e^{\sigma}}^{\bullet}(R_e^{\sigma,\theta \mhyphen ss}) \rightarrow H_{\mathsf{G}_e^{\sigma}}^{\bullet}(R_e^{\sigma,\theta \mhyphen st})$ factors through a surjection $H_{\mathsf{G}_e^{\sigma}}^{\bullet}(R_e^{\sigma,\theta \mhyphen ss}) \twoheadrightarrow PH_{\mathsf{G}_e^{\sigma}}^{\bullet}(R_e^{\sigma, \theta \mhyphen st})$. This is carried out in \cite[Proposition 3.9]{mbyoung2016b} under the assumption that $Q$ is $\sigma$-symmetric and $\theta=0$, but the same argument works in general. The only new ingredient is that $H_{\mathsf{G}_e^{\sigma}}^{\bullet}(R_e^{\sigma,\theta \mhyphen ss})$ vanishes in odd degrees and is of Hodge-Tate type otherwise, which was proved in Proposition \ref{prop:noOddCohom}. For each $k \geq 0$ we therefore obtain an exact commutative diagram
\[
\begin{tikzpicture}
  \matrix (m) [matrix of math nodes,nodes={anchor=center},row sep=3em,column sep=3em,minimum width=1.4em] {
A_{\mathsf{G}_e^{\sigma}}^k(R_e^{\sigma,\theta \mhyphen ss})_{\mathbb{Q}} & A_{\mathsf{G}_e^{\sigma}}^k(R_e^{\sigma,\theta \mhyphen st})_{\mathbb{Q}} \\
H_{\mathsf{G}_e^{\sigma}}^{2k}(R_e^{\sigma,\theta \mhyphen ss}) & P H_{\mathsf{G}_e^{\sigma}}^{2k}(R_e^{\sigma,\theta \mhyphen st})\\};
\draw  (m-1-1) edge[->>] (m-1-2);
\draw  (m-1-1) edge[->] node[left]{$\mathsf{cl}$} (m-2-1);
\draw  (m-2-1) edge[->>] (m-2-2);
\draw  (m-1-2) edge[->] node[right]{$\mathsf{cl}$} (m-2-2);
\end{tikzpicture}
\]
By Theorem \ref{thm:moduleComparison} the left-hand vertical map is an isomorphism. Surjectivity of the right-hand vertical map follows. To complete the proof, since $\mathbf{M}_e^{\sigma, \theta \mhyphen st}$ is a Deligne-Mumford stack we can apply \cite[Theorem 4]{edidin1998}  to get isomorphisms
\[
A^{\bullet}_{\mathsf{G}_e^{\sigma}}(R_e^{\sigma, \theta \mhyphen st})_{\mathbb{Q}} \simeq A^{\bullet}(\mathbf{M}_e^{\sigma, \theta \mhyphen st})_{\mathbb{Q}} \simeq A^{\bullet}(\mathfrak{M}_e^{\sigma, \theta \mhyphen st})_{\mathbb{Q}}
\]
and similarly for cohomology groups with their mixed Hodge structures.
\end{proof}

In the setting of ordinary quiver representations Corollary \ref{cor:purePart} can be strengthened as follows.

\begin{Thm}
\label{thm:purePartOrdinary}
Let $Q$ be a quiver with stability $\theta$. For each $d \in \Lambda_Q^+$ the cycle map $\mathsf{cl}: A^{\bullet}(\mathfrak{M}_d^{\theta \mhyphen st})_{\mathbb{Q}} \rightarrow H^{\bullet}(\mathfrak{M}_d^{\theta \mhyphen st})$ surjects onto the pure part. Moreover, if $Q$ is the double of a quiver, then $\mathsf{cl}: A^{\bullet}(\mathfrak{M}_d^{\theta \mhyphen st})_{\mathbb{Q}} \rightarrow H^{\bullet}(\mathfrak{M}_d^{\theta \mhyphen st})$ is an isomorphism.
\end{Thm}

\begin{proof}
The first statement is proved in the same way as Corollary \ref{cor:purePart}. So assume that $Q$ is a double quiver. In the case of trivial stability it is proved in \cite[Theorem 2.2]{chen2014} that the restriction map $H_{\mathsf{GL}_d}^{\bullet}(R_d) \rightarrow H_{\mathsf{GL}_d}^{\bullet}(R_d^{st})$ induces an isomorphism $V^{\mathsf{prim}}_{Q,d} \xrightarrow[]{\sim} PH^{\bullet}(\mathfrak{M}_d^{st})\{\chi(d,d) \slash 2 \}$, where $V^{\mathsf{prim}}_{Q,d}$ denotes the cohomological Donaldson-Thomas invariant of $Q$ \cite{efimov2012}. On the other hand, it is proved in \cite[\S 9]{franzen2015b} that the cycle map defines an isomorphism $A^{\bullet}(\mathfrak{M}^{st}_d)_{\mathbb{Q}}\{\chi(d,d) 
\slash 2 \} \xrightarrow[]{\sim} V^{\mathsf{prim}}_{Q,d}$, where $A^k(\mathfrak{M}^{st}_d)_{\mathbb{Q}}$ is given $\mathbb{Z}$-degree $2k$.

For an arbitrary stability $\theta$, the open embedding $\mathfrak{M}_d^{st} \hookrightarrow \mathfrak{M}_d^{\theta \mhyphen st}$ induces a commutative diagram
\[
\begin{tikzpicture}
  \matrix (m) [matrix of math nodes,nodes={anchor=center},row sep=3em,column sep=3em,minimum width=1.4em] {
A^k(\mathfrak{M}_d^{\theta \mhyphen st})_{\mathbb{Q}} & A^k(\mathfrak{M}_d^{st})_{\mathbb{Q}} \\
PH^{2k}(\mathfrak{M}_d^{\theta \mhyphen st}) & PH^{2k}(\mathfrak{M}_d^{st}) \\};
\draw  (m-1-1) edge[->>] (m-1-2);
\draw  (m-1-1) edge[->>] node[left]{$\mathsf{cl}$} (m-2-1);
\draw  (m-2-1) edge[->] (m-2-2);
\draw  (m-1-2) edge[->] node[right]{$\mathsf{cl}$} (m-2-2);
\end{tikzpicture}
\]
whose top horizontal map is an isomorphism by \cite[Theorem 9.2]{franzen2015b} and whose right vertical map is an isomorphism by the discussion above. This implies that the left vertical map is injective and hence is an isomorphism.
\end{proof}

\subsection{Restriction to $\sigma$-stable representations}

In this section we relate $\mathcal{B}_Q^{\theta \mhyphen ss}$ to the rational Chow groups of moduli spaces of $\sigma$-stable self-dual representations.

Define
\[
\mathcal{B}_Q^{\theta \mhyphen st} = \bigoplus_{e \in \Lambda_Q^{\sigma,+}} A^{\bullet}_{\mathsf{G}_e^{\sigma}}(R_e^{\sigma, \theta \mhyphen st})_{\mathbb{Q}}\{ \mathcal{E}(e) \slash 2\}  \in D^{lb}(\mathsf{Vect}_{\mathbb{Z}})_{\Lambda_Q^{\sigma,+}}.
\]
The open inclusions $R_e^{\sigma, \theta \mhyphen st} \hookrightarrow R_e^{\sigma, \theta \mhyphen ss}$ induce a surjection $\mathcal{B}_Q^{\theta \mhyphen ss} \twoheadrightarrow \mathcal{B}_Q^{\theta \mhyphen st}$.

\begin{Prop}
\label{prop:kernelRestriction}
For each $e \in \Lambda_Q^{\sigma,+}$ the kernel of the restriction map $\mathcal{B}_{Q,e}^{\theta \mhyphen ss} \twoheadrightarrow \mathcal{B}_{Q,e}^{\theta \mhyphen st}$ is equal to
\[
\sum_{\substack{(d^{\prime},e^{\prime}) \in \Lambda_{Q,\mu=0}^+ \times \Lambda_Q^{\sigma,+} \\ H(d^{\prime}) + e^{\prime} = e, \, d^{\prime} \neq 0}} \mathcal{A}_{Q,d^{\prime}}^{\theta \mhyphen ss} \star \mathcal{B}_{Q,e^{\prime}}^{\theta \mhyphen ss}.
\]
\end{Prop}

\begin{proof}

	By definition, the locus $R_e^{\sigma, \theta \mhyphen ss} - R_e^{\sigma, \theta \mhyphen st}$ of properly semistable self-dual representations is the union
	$$
		\bigcup_{\substack{(d^{\prime},e^{\prime}) \in \Lambda_{Q,\mu=0}^+ \times \Lambda_Q^{\sigma,+} \\ H(d^{\prime}) + e^{\prime} = e, \, d^{\prime} \neq 0}} R_e^{\sigma, (d^{\prime}, e^{\prime}),  \theta \mhyphen ss}
	$$
	of the subsets $R_e^{\sigma, (d^{\prime}, e^{\prime}),  \theta \mhyphen ss}$ of semistable self-dual representations of dimension vector $e$ which possess an isotropic subrepresentation of dimension vector $d^{\prime}$. From the proof of Proposition \ref{prop:sdHNStrat} we know that $R_e^{\sigma, (d^{\prime}, e^{\prime}),  \theta \mhyphen ss}$ is the $\mathsf{G}_e^{\sigma}$-saturation of $R_{d^{\prime}, e^{\prime}}^{\sigma, \theta \mhyphen ss}$ and that $R_{d^{\prime}, e^{\prime}}^{\sigma, \theta \mhyphen ss}$ is a closed subset of $R_e^{\sigma, \theta \mhyphen ss}$. We can therefore apply \cite[Lemma 8.2]{franzen2015b} to conclude that
	\begin{equation} \label{eqn}
		\bigoplus_{\substack{(d^{\prime},e^{\prime}) \in \Lambda_{Q,\mu=0}^+ \times \Lambda_Q^{\sigma,+} \\ H(d^{\prime}) + e^{\prime} = e, \, d^{\prime} \neq 0}} A_k^{\mathsf{G}_e^{\sigma}}(\mathsf{G}_e^{\sigma} \times_{\mathsf{G}_{d',e'}^\sigma} R_{d^{\prime}, e^{\prime}}^{\sigma, \theta \mhyphen ss} )_{\mathbb{Q}} \to A_k^{\mathsf{G}_e^{\sigma}}(R_e^{\sigma, \theta \mhyphen ss})_{\mathbb{Q}} \to A_k^{\mathsf{G}_e^{\sigma}}(R_e^{\sigma, \theta \mhyphen st})_{\mathbb{Q}} \to 0
	\end{equation}
	is an exact sequence. The affine bundles
	$$
		\mathsf{G}_e^\sigma \times_{\mathsf{GL}_{d'} \times \mathsf{G}_{e'}^\sigma} (R_{d'}^{\theta \mhyphen ss} \times R_{e'}^{\sigma, \theta \mhyphen ss}) \leftarrow  \mathsf{G}_e^\sigma \times_{\mathsf{GL}_{d'} \times \mathsf{G}_{e'}^\sigma} R_{d^{\prime}, e^{\prime}}^{\sigma, \theta \mhyphen ss} \to \mathsf{G}_e^\sigma  \times_{\mathsf{G}_{d',e'}^\sigma} R_{d^{\prime}, e^{\prime}}^{\sigma, \theta \mhyphen ss}
	$$
	give rise to isomorphisms
	\begin{align*}
		A_k^{\mathsf{G}_e^\sigma}( \mathsf{G}_e^{\sigma} \times_{\mathsf{G}_{d',e'}^\sigma} R_{d^{\prime}, e^{\prime}}^{\sigma, \theta \mhyphen ss} ) &\simeq A_{\mathsf{G}_{d',e'}^\sigma}^{2\dim R_e^{\sigma} - k +\chi(d',e')+\mathcal{E}(\sigma(d'))}(R_{d^{\prime}, e^{\prime}}^{\sigma, \theta \mhyphen ss}) \\
		&\simeq A_{\mathsf{GL}_{d'} \times \mathsf{G}_{e'}^\sigma}^{2\dim R_e^{\sigma} - k +\chi(d',e')+\mathcal{E}(\sigma(d'))}(R_{d'}^{\theta \mhyphen ss} \times R_{e'}^{\sigma,\theta \mhyphen ss}).
	\end{align*}
It was shown in the proof of Theorem \ref{thm:moduleComparison} that the exterior product map
	$$
	A_{\mathsf{GL}_{d'}}^{\bullet}(R_{d'}^{\theta \mhyphen ss})_{\mathbb{Q}} \otimes A_{\mathsf{G}_{e'}^\sigma}^{\bullet}(R_{e'}^{\sigma,\theta \mhyphen ss})_{\mathbb{Q}} \to A_{\mathsf{GL}_{d'} \times \mathsf{G}_{e'}^\sigma}^{\bullet}(R_{d'}^{\theta \mhyphen ss} \times R_{e'}^{\sigma,\theta \mhyphen ss})_{\mathbb{Q}}
	$$
	 is an isomorphism. Combining these isomorphisms and taking into account the grading shifts, the leftmost map of the sequence (\ref{eqn}) identifies with the ChowHA action map
	 \[	  \bigoplus_{\substack{(d^{\prime},e^{\prime}) \in \Lambda_{Q,\mu=0}^+ \times \Lambda_Q^{\sigma,+} \\ H(d^{\prime}) + e^{\prime} = e, \, d^{\prime} \neq 0}} \mathcal{A}_{Q,d^{\prime}}^{\theta \mhyphen ss} \boxtimes^{S \mhyphen \mathsf{tw}} \mathcal{B}_{Q,e^{\prime}}^{\theta \mhyphen ss} \xrightarrow[]{\star} \mathcal{B}_{Q,e}^{\theta \mhyphen ss}.
	 \]
This completes the proof.
\end{proof}

\subsection{Chow theoretic wall-crossing formulas}

In \cite{mbyoung2015} a motivic orientifold wall-crossing formula was proved using finite field Hall algebras and their representations, generalizing and explaining formulas which appeared previously in the string theory literature \cite{denef2010}. In this section we lift the wall-crossing formula to Chow theoretic and cohomological Hall modules.

By applying \cite[Corollary 5.4]{franzen2015b} to the $\sigma$-HN stratification of $R_e^{\sigma}$ we obtain an isomorphism 
\begin{equation}
\label{eq:hnDecomp}
	A_{\mathsf{G}_e^\sigma}^\bullet(R_e^\sigma) \simeq \bigoplus_{(d^\bullet,e^\infty) \in \mathsf{HN}^\sigma(e)} A_{\mathsf{G}_e^\sigma}^{\bullet-\codim_{R_e^{\sigma}}(R_{d^\bullet,e^\infty}^{\sigma,HN})}(R_{d^\bullet,e^\infty}^{\sigma,HN}).
\end{equation}
Using the isomorphism \eqref{eq:hnStrataIsom} we see that
$$
	\codim_{R_e^\sigma}(R_{d^\bullet,e^\infty}^{\sigma,HN}) = \codim_{R_e^\sigma}(R_{d^\bullet,e^\infty}^\sigma) - \dim \mathsf{G}_e^\sigma + \dim \mathsf{G}_{d^\bullet,e^\infty}.
$$
By displaying $R_{d^\bullet,e^\infty}^\sigma$ in terms of matrices, analogous to the description of $R_{d,e}^\sigma$ given in \cite[\S 3.1]{mbyoung2016b}, we compute
\begin{equation}
\label{eq:codimStrata}
\codim_{R_e^{\sigma}}(R_{d^\bullet,e^\infty}^{\sigma,HN})=	- \sum_{1 \leq k < l \leq r} \chi(d^k,d^l) - \chi(d,e^\infty) - \mathcal{E}(\sigma(d))
\end{equation}
where we have set $d=d^1+ \cdots + d^r$.

In particular, by considering the open stratum $R_e^{\sigma, \theta \mhyphen ss} \subset R_e^{\sigma}$ we obtain from \eqref{eq:hnDecomp} a vector space splitting of the surjection $\mathcal{B}_Q \twoheadrightarrow \mathcal{B}^{\theta \mhyphen ss}_Q$. We use this splitting to regard $\mathcal{B}_Q^{\theta \mhyphen ss}$ as a subobject of $\mathcal{B}_Q$. In the same way, we can consider $\mathcal{A}_Q^{\theta \mhyphen ss}$ as a subobject of $\mathcal{A}_Q$.

\begin{Thm}
\label{thm:hnFactor}
Let $\theta$ be a $\sigma$-compatible stability. Then the slope ordered ChowHA action map
\[
\overset{\hspace{-2.5em}\longleftarrow}{\boxtimes^{\mathsf{tw}}_{\mu \in \mathbb{Q}_{> 0}}} \mathcal{A}^{\theta \mhyphen ss}_{Q, \mu} \boxtimes^{S \mhyphen \mathsf{tw}} \mathcal{B}^{\theta \mhyphen ss}_Q \xrightarrow[]{\star} \mathcal{B}_Q
\]
is an isomorphism in $D^{lb}(\mathsf{Vect}_{\mathbb{Z}})_{\Lambda_Q^{\sigma,+}}$. The analogous statement for $\mathcal{M}_Q$ also holds.
\end{Thm}

\begin{proof}
Identify $A_{\mathsf{G}_e^\sigma}^\bullet(R_{d^\bullet,e^\infty}^{\sigma,HN})_{\mathbb{Q}}$ with the tensor product $A_{\mathsf{GL}_{d^1}}^\bullet(R_{d^1}^{\theta \mhyphen ss})_{\mathbb{Q}} \otimes \cdots \otimes A_{\mathsf{GL}_{d^r}}^\bullet(R_{d^r}^{\theta \mhyphen ss})_{\mathbb{Q}} \otimes A_{\mathsf{G}_{e^\infty}^\sigma}^\bullet(R_{e^\infty}^{\sigma,\theta \mhyphen ss})_{\mathbb{Q}}$ and consider the sections of the surjections
$$
	\Big( \bigotimes_{k=1}^r A_{\mathsf{GL}_{d^k}}^\bullet(R_{d^k})_{\mathbb{Q}} \Big) \otimes A_{\mathsf{G}_{e^\infty}^\sigma}^\bullet(R_{e^\infty}^{\sigma})_{\mathbb{Q}} \twoheadrightarrow \Big( \bigotimes_{k=1}^r A_{\mathsf{GL}_{d^k}}^\bullet(R_{d^k}^{\theta \mhyphen ss})_{\mathbb{Q}} \Big) \otimes A_{\mathsf{G}_{e^\infty}^\sigma}^\bullet(R_{e^\infty}^{\sigma,\theta \mhyphen ss})_{\mathbb{Q}}
$$
coming from the $\sigma$-HN stratification. This leads to a commutative diagram
\begin{equation}
	\label{d:multVsHN}
	\begin{tikzpicture}
		\matrix (m) [matrix of math nodes,nodes={anchor=center},row sep=2em,column sep=2em,minimum width=1.2em] {
		\displaystyle \Big( \bigotimes_{k=1}^r A_{\mathsf{GL}_{d^k}}^\bullet(R_{d^k}^{\theta \mhyphen ss})_{\mathbb{Q}} \Big) \otimes A_{\mathsf{G}_{e^\infty}^\sigma}^\bullet(R_{e^\infty}^{\sigma,\theta \mhyphen ss})_{\mathbb{Q}} & A_{\mathsf{G}_e^\sigma}^\bullet(R_{d^\bullet,e^\infty}^{\sigma,HN})_{\mathbb{Q}} \\
		\displaystyle \Big( \bigotimes_{k=1}^r A_{\mathsf{GL}_{d^k}}^\bullet(R_{d^k})_{\mathbb{Q}} \Big) \otimes A_{\mathsf{G}_{e^\infty}^\sigma}^\bullet(R_{e^\infty}^{\sigma})_{\mathbb{Q}} & A_{\mathsf{G}_e^\sigma}^{\bullet-\codim_{R_e^{\sigma}}(R_{d^\bullet,e^\infty}^{\sigma,HN})}(R_{d^\bullet,e^\infty}^{\sigma})_{\mathbb{Q}}  \\
		};
		\draw  (m-1-1) edge[->] node[above]{$\cong$} (m-1-2);
		\draw  (m-1-1) edge[->] (m-2-1);
		\draw  (m-2-1) edge[->] node[below]{$\star$} (m-2-2);
		\draw  (m-1-2) edge[->] (m-2-2);
	\end{tikzpicture}
\end{equation}
The vertical maps are the sections arising from the $\sigma$-HN stratifications. Comparing equation \eqref{eq:codimStrata} with the twists appearing in the monoidal products $\boxtimes^{\mathsf{tw}}$ and $\boxtimes^{S \mhyphen \mathsf{tw}}$, we conclude that the ChowHA action induces the desired isomorphism $\overset{\hspace{-2.5em}\longleftarrow}{\boxtimes^{\mathsf{tw}}_{\mu \in \mathbb{Q}_{> 0}}} \mathcal{A}^{\theta \mhyphen ss}_{Q, \mu} \boxtimes^{S \mhyphen \mathsf{tw}} \mathcal{B}^{\theta \mhyphen ss}_Q \xrightarrow[]{\star} \mathcal{B}_Q$. By Theorem \ref{thm:moduleComparison} the cohomological statement now also follows.
\end{proof}

If $Q$ is finite type or a cyclic affine $A_1$ quiver, a proof of Theorem \ref{thm:hnFactor} via a direct study of $\mathcal{M}_Q$ is given in \cite[Theorems 4.9 and 5.8]{mbyoung2016b}.

\section{Applications to orientifold Donaldson-Thomas theory}
\label{sec:oriDT}

\subsection{Orientifold DT invariants are Chow groups}

\label{sec:oriDTChow}

Recall that a quiver $Q$ is called symmetric if its Euler form $\chi$ is a symmetric bilinear form. If $Q$ is symmetric, then $\boxtimes^{\mathsf{tw}}$ reduces to the untwisted monoidal product and $\mathcal{A}_{Q,\mu}^{\theta \mhyphen ss}$ becomes a $\Lambda^+_{Q,\mu} \times \mathbb{Z}$-graded algebra. Moreover, up to a twist of the multiplication by a sign, the algebra $\mathcal{A}^{\theta \mhyphen ss}_{Q,\mu}$ is supercommutative \cite[\S 2.6]{kontsevich2011}.

Similarly, a quiver with involution and duality structure is called $\sigma$-symmetric if it is symmetric and $\sigma^* \mathcal{E} = \mathcal{E}$. In this case $\boxtimes^{S \mhyphen \mathsf{tw}}$ reduces to the untwisted monoidal module structure and $\mathcal{B}_{Q}^{\theta \mhyphen ss}$ is a $\Lambda_Q^{\sigma,+} \times \mathbb{Z}$-graded $\mathcal{A}_{Q,\mu=0}^{\theta \mhyphen ss}$-module. If $\mathcal{A}^{\theta \mhyphen ss}_{Q,\mu=0}$ is supercommutative without any twist, then $\mathcal{B}^{\theta \mhyphen ss}_Q$ is a super $\mathcal{A}^{\theta \mhyphen ss}_{Q,\mu=0}$-module. In general it is unknown if the supercommutative twist of $\mathcal{A}^{\theta \mhyphen ss}_{Q,\mu=0}$ can be lifted to $\mathcal{B}^{\theta \mhyphen ss}_Q$. In any case, we will not use the supercommutative twist in what follows.

Let $\mathcal{H}^{\theta \mhyphen ss}_{Q, \mu=0,+}$ be the augmentation ideal of $\mathcal{H}^{\theta \mhyphen ss}_{Q, \mu=0}$. In \cite{mbyoung2016b} the cohomological orientifold Donaldson-Thomas invariant of a $\sigma$-symmetric quiver was defined to be the $\Lambda_Q^{\sigma,+} \times \mathbb{Z}$-graded vector space
\[
W^{\mathsf{prim}, \theta}_Q= \mathcal{M}^{\theta \mhyphen ss}_Q \slash ( \mathcal{H}^{\theta \mhyphen ss}_{Q, \mu=0,+} \star \mathcal{M}^{\theta \mhyphen ss}_Q ).
\]
Denote by $W^{\mathsf{prim},\theta}_{Q,(e,l)}$ the degree $(e,l) \in \Lambda_Q^{\sigma,+} \times \mathbb{Z}$ summand of $W^{\mathsf{prim}}_Q$.
 
\begin{Thm}
\label{thm:sdChowInterpretation}
Let $Q$ be a $\sigma$-symmetric quiver with $\sigma$-compatible stability $\theta$. Then
\[
W^{\mathsf{prim},\theta}_{Q,(e,l)} \simeq \left\{
\begin{array}{ll}
A_{-\frac{1}{2}(l+\mathcal{E}(e))}(\mathfrak{M}_e^{\sigma, \theta \mhyphen st})_{\mathbb{Q}} & \mbox{ if } l+\mathcal{E}(e) \equiv 0 \mod 2,\\
0 & \mbox{ if } l+\mathcal{E}(e) \equiv 1\mod 2.
\end{array}
\right.
\]
In particular, if $W^{\mathsf{prim},\theta}_{Q,(e,l)}$ is non-trivial, then $\mathcal{E}(e) \leq l \leq -\mathcal{E}(e)$.
\end{Thm}

\begin{proof}
By Theorem \ref{thm:moduleComparison} the equivariant cycle map induces an isomorphism
\[
\mathcal{B}^{\theta \mhyphen ss}_Q \slash  (\mathcal{A}^{\theta \mhyphen ss}_{Q, \mu=0,+} \star \mathcal{B}^{\theta \mhyphen ss}_Q) \xrightarrow[]{\sim} \mathcal{M}^{\theta \mhyphen ss}_Q \slash (\mathcal{H}^{\theta \mhyphen ss}_{Q, \mu=0,+} \star \mathcal{M}^{\theta \mhyphen ss}_Q)
\]
in $D^{lb}(\mathsf{Vect}_{\mathbb{Z}})_{\Lambda_Q^{\sigma,+}}$. Proposition \ref{prop:kernelRestriction} implies that
\begin{equation}
\label{eq:stableQuot}
\mathcal{B}_Q^{\theta \mhyphen st} \simeq \mathcal{B}^{\theta \mhyphen ss}_Q \slash ( \mathcal{A}^{\theta \mhyphen ss}_{Q, \mu=0,+} \star \mathcal{B}^{\theta \mhyphen ss}_Q ).
\end{equation}
It follows that we have an induced isomorphism $\mathcal{B}_Q^{\theta \mhyphen st} \xrightarrow[]{\sim} W^{\mathsf{prim},\theta}_Q$. By definition, the degree $(e, l) \in \Lambda_Q^{\sigma,+} \times \mathbb{Z}$ component $\mathcal{B}_{Q,(e,l)}^{\theta \mhyphen st}$ of $\mathcal{B}_Q^{\theta \mhyphen st}$ is trivial unless $l- \mathcal{E}(e)$ is even, in which case we have
\begin{eqnarray*}
\mathcal{B}_{Q,(e,l)}^{\theta \mhyphen st} &=& A^{\frac{1}{2}(l - \mathcal{E}(e))}_{\mathsf{G}_e^{\sigma}}(R_e^{\sigma, \theta \mhyphen st})_{\mathbb{Q}} \\
& \simeq & A^{\frac{1}{2}(l - \mathcal{E}(e))}(\mathbf{M}_e^{\sigma, \theta \mhyphen st})_{\mathbb{Q}} \\
& \simeq & A^{\frac{1}{2}(l - \mathcal{E}(e))}(\mathfrak{M}_e^{\sigma, \theta \mhyphen st})_{\mathbb{Q}} \\
& \simeq & A_{-\frac{1}{2}(l + \mathcal{E}(e))}(\mathfrak{M}_e^{\sigma, \theta \mhyphen st})_{\mathbb{Q}}.
\end{eqnarray*}
These isomorphisms follow from \cite[Theorem 4]{edidin1998} together with the fact that, if non-empty, the complex dimension of $\mathfrak{M}_e^{\sigma, \theta \mhyphen st}$ is $-\mathcal{E}(e)$. The final statement of the theorem also follows from this dimension formula.
\end{proof}

\begin{Cor}
\label{cor:intConj}
For each $e \in \Lambda_Q^{\sigma,+}$ the vector space $\bigoplus_{l \in \mathbb{Z}}W^{\theta, \mathsf{prim}}_{Q,(e,l)}$ is finite dimensional.
\end{Cor}

The statement of Corollary \ref{cor:intConj} is known as the orientifold integrality conjecture. A direct proof in the case of trivial stability was given in \cite[Theorem 3.4]{mbyoung2016b}.

The motivic orientifold Donaldson-Thomas invariant of $Q$ is defined by
\[
\Omega_Q^{\sigma, \theta} (q^{\frac{1}{2}}, \xi) = \sum_{(e,l) \in \Lambda_Q^{\sigma,+} \times \mathbb{Z}}  \dim_{\mathbb{Q}} W_{Q,(e,l)}^{\mathsf{\mathsf{prim},\theta}} (-q^{\frac{1}{2}})^l \xi^e \in \mathbb{Z}[q^{\frac{1}{2}}, q^{-\frac{1}{2}}] \pser{ \xi }.
\]
Hence $\Omega_Q^{\sigma, \theta}$ is the generating series of shifted Chow theoretic Poincar\'{e} polynomials of $\mathfrak{M}_e^{\sigma, \theta \mhyphen st}$. If we define a normalized invariant by
\[
\overline{\Omega}_{Q,e}^{\sigma, \theta} = (-q^{\frac{1}{2}})^{-\mathcal{E}(e)}
\Omega_{Q,e}^{\sigma, \theta},
\]
then $\overline{\Omega}_{Q,e}^{\sigma, \theta} \in \mathbb{Z}_{\geq 0}[q]$ with constant term $1$ and degree at most $-\mathcal{E}(e)$.

\begin{Ex}
A symmetric quiver $Q$ admits an essentially unique involution $\sigma$ which restricts to the identity on $Q_0$. Any duality structure on $(Q, \sigma)$ is automatically $\sigma$-symmetric. The only $\sigma$-compatible stability is the trivial stability. In this setting we have $\overline{\Omega}_{Q,e}^{\sigma} \in \mathbb{Z}_{\geq 0}[q^2]$. Using Theorem \ref{thm:sdChowInterpretation} and Corollary \ref{cor:purePart} this implies that
\[
A_{-\mathcal{E}(e) -2k-1}(\mathfrak{M}_e^{\sigma,st})_{\mathbb{Q}} =0
\]
for all $k \in \mathbb{Z}$ and
\[
PH^k(\mathfrak{M}_e^{\sigma,st}) =0
\]
unless $k \equiv 0 \mod 4$.
\end{Ex}

\begin{Rem}
Without the assumption of $\sigma$-symmetry, Proposition \ref{prop:kernelRestriction} still implies equation \eqref{eq:stableQuot}. However in this generality we do not expect the right-hand side to be the correct definition of the orientifold DT invariant.
\end{Rem}

\begin{Thm}
\label{thm:noWallCrossing}
Let $Q$ be a $\sigma$-symmetric quiver. For any $\sigma$-compatible stabilities $\theta$ and $\theta^{\prime}$ there is a canonical isomorphism $W_Q^{\mathsf{prim}, \theta^{\prime}} \simeq W_Q^{\mathsf{prim}, \theta}$ of $\Lambda_Q^{\sigma,+} \times \mathbb{Z}$-graded vector spaces. In particular, the motivic orientifold Donaldson-Thomas invariant of a $\sigma$-symmetric quiver is independent of the choice of $\sigma$-compatible stability.
\end{Thm}

\begin{proof}
It suffices to prove the theorem in the case $\theta^{\prime} = 0$. Let $\rho : \mathcal{B}_Q \twoheadrightarrow \mathcal{B}_Q^{\theta \mhyphen ss}$ be the restriction map. By Theorem \ref{thm:moduleComparison} together with the surjectivity of the restriction map $\mathcal{A}_{Q,\mu=0} \twoheadrightarrow \mathcal{A}_{Q,\mu=0}^{\theta \mhyphen ss}$ we have a commutative diagram
\[
\begin{tikzpicture}
  \matrix (m) [matrix of math nodes,nodes={anchor=center},row sep=2em,column sep=2em,minimum width=1.4em] {
\mathcal{A}_{Q, \mu=0,+} \star \mathcal{B}_Q & \mathcal{A}^{\theta \mhyphen ss}_{Q, \mu=0,+} \star \mathcal{B}^{\theta \mhyphen ss}_Q \\
\mathcal{A}_{Q,+} \star \mathcal{B}_Q &  \\
\mathcal{B}_Q & \mathcal{B}_Q^{\theta \mhyphen ss} \\
};
\draw  (m-1-1) edge[->>] (m-1-2);
\draw [right hook->] (m-1-1) -- (m-2-1);
\draw [right hook->] (m-1-2) -- (m-3-2);
\draw [right hook->] (m-2-1) -- (m-3-1);
\draw  (m-3-1) edge[->>] node[below]{$\rho$} (m-3-2);
\end{tikzpicture}
\]

We have a decomposition
\[
\mathcal{A}_Q = \bigoplus_{\mu \in \mathbb{Q}} \mathcal{A}_{Q, \mu}
\]
as $\mathbb{Z}$-graded vector spaces. Suppose that $x \in \mathcal{A}_{Q,\mu,+}$ and $\xi \in \mathcal{B}_Q$ are homogeneous. If $\mu >0$, then $\rho( x \star \xi) =0$. Indeed, if $d \in \Lambda_{Q,\mu}^+$, then the action map $\mathcal{A}_{Q,d} \boxtimes \mathcal{B}_{Q,e} \rightarrow \mathcal{B}_{Q,H(d)+e}$ factors through the map $A_{\mathsf{G}_{H(d)+e}^{\sigma}}^{\bullet}(R_{H(d)+e}^{\sigma,(d,e)})_{\mathbb{Q}} \rightarrow A_{\mathsf{G}_{H(d)+e}^{\sigma}}^{\bullet}(R_{H(d)+e}^{\sigma})_{\mathbb{Q}}$, whose image is annihilated by $\rho$. If instead $\mu <0$, then $S_{\mathcal{A}}(x) \in \mathcal{A}_{Q,-\mu,+}$ where $S_{\mathcal{A}}: \mathcal{A}_Q \rightarrow \mathcal{A}_Q$ is the algebra anti-involution induced by the duality structure. Hence $\rho (S_{\mathcal{A}}(x) \star \xi) =0$. But $S_{\mathcal{A}}(x) \star \xi$ and $x \star \xi$ differ by sign \cite[Proposition 3.5]{mbyoung2016b}, so again we have $\rho(x \star \xi)=0$.

It follows that $\rho$ descends to a surjection $W^{\mathsf{prim}}_Q  \twoheadrightarrow W^{\mathsf{prim},\theta}_Q$ of graded vector spaces. To see that this map is also injective, suppose that $\xi \in \mathcal{B}_{Q,e}$ satisfies
\[
\rho(\xi) \in \mathcal{A}_{Q,\mu=0,+}^{\theta \mhyphen ss} \star \mathcal{B}_Q^{\theta \mhyphen ss}.
\]
We need to show that $\xi \in \mathcal{A}_{Q,+} \star \mathcal{B}_Q$. Since $\mathcal{A}_{Q,\mu=0} \rightarrow \mathcal{A}^{\theta \mhyphen ss}_{Q,\mu=0}$ is surjective, there exists an element $\xi^{\prime} \in \mathcal{A}_{Q,+} \star \mathcal{B}_Q$ such that $\xi - \xi^{\prime} \in \ker (\rho)$. By the decomposition \eqref{eq:hnDecomp} we therefore have
\[
\xi - \xi^{\prime} \in \bigoplus_{(d^{\bullet},e^{\infty}) \in \mathsf{HN}^{\sigma}(e) \backslash (e)} A_{\mathsf{G}_e^\sigma}^{\bullet-\codim_{R_e^{\sigma}}(R_{d^\bullet,e^\infty}^{\sigma,HN})}(R_{d^\bullet,e^\infty}^{\sigma,HN}) 
\]
which by the commutative diagram (\ref{d:multVsHN}) (in the proof of Theorem \ref{thm:hnFactor}) is contained in $\mathcal{A}_{Q,+} \star \mathcal{B}_Q$.
It follows that $\xi \in \mathcal{A}_{Q,+} \star \mathcal{B}_Q$.
\end{proof}

Combining Theorems \ref{thm:sdChowInterpretation} and \ref{thm:noWallCrossing} we see that the rational Chow groups of $\mathfrak{M}_e^{\sigma,\theta \mhyphen st}$ are independent of $\theta$.

\subsection{Orientifold DT invariants of  loop quivers}
\label{sec:loopQuiverApplication}

Let $L_m$ be the quiver with one node and $m \geq 0$ loops. Then $L_m$ admits a unique involution, being the identity on both nodes and arrows, and we are in the setting of the example from Section \ref{sec:oriDTChow}. Any stability is equivalent to the trivial stability, which is $\sigma$-compatible. A duality structure on $L_m$ is given by a sign $s$ and $m$ signs $\tau$. Let $\tau_+$ (respectively, $\tau_-$) be the number of the latter which are positive (negative).

Suppose that $\tau$ is identically $-1$, that is $\tau_- = m$. If $s=1$, then the variety of self-dual representations is $R_e^{\sigma} = \mathfrak{so}_e(\mathbb{C})^{\oplus m}$ with the simultaneous adjoint action of $\mathsf{G}_e^{\sigma} = \mathsf{O}_e(\mathbb{C})$ while if $s=-1$, then $R_e^{\sigma} = \mathfrak{sp}_e(\mathbb{C})^{\oplus m}$ with the simultaneous adjoint action of $\mathsf{G}_e^{\sigma} = \mathsf{Sp}_e(\mathbb{C})$. Hence $\mathfrak{M}_e^{\sigma, st}$ is a moduli space of stable $m$-tuples in a classical Lie algebra. When $s=1$ we have a decomposition $\mathcal{B}_{L_m} = \mathcal{B}^D_{L_m} \oplus \mathcal{B}_{L_m}^B$, the summands corresponding to even and odd dimensional representations, respectively.

\begin{Thm}
\label{thm:orthSymDuality}
For any $m \geq 0$ and $e \geq 0$ we have an isomorphism
\[
A_{\bullet}((\mathfrak{sp}_{2e}^{\oplus m})^{st} \slash \mathsf{Sp}_{2e})_{\mathbb{Q}} \simeq A_{\bullet}((\mathfrak{so}_{2e+1}^{\oplus m})^{st} \slash \mathsf{O}_{2e+1})_{\mathbb{Q}}
\]
of $\mathbb{Z}$-graded rational vector spaces.
\end{Thm}

\begin{proof}
Define a $\Lambda_{L_m}^+ \times \mathbb{Z}$-graded ring automorphism
\[
\phi: \mathcal{H}_{L_m} \rightarrow \mathcal{H}_{L_m}
\qquad
\mathcal{H}_{L_m,d} \owns f_d \mapsto 2^d f_d.
\]
Let $(\mathcal{M}_{L_m}^B)_{\phi}$ be the $\phi$-twisted module associated to $\mathcal{M}_{L_m}^B$. Explicitly, $(\mathcal{M}_{L_m}^B)_{\phi}$ equals $\mathcal{M}_{L_m}^B$ as a graded abelian group and has $\mathcal{H}_{L_m}$-module structure
\[
f_d \star_{\phi} g = \phi(f_d) \star g.
\]
Comparing the explicit signed shuffle descriptions of $\mathcal{M}_{L_m}^B$ and $\mathcal{M}_{L_m}^C$ given in \cite[Theorem 3.3]{mbyoung2016b}, we see immediately that $(\mathcal{M}_{L_m}^B)_{\phi} \simeq \mathcal{M}_{L_m}^C [1]$ via the identity map, where $[1]$ denotes $\Lambda_{L_m}^{\sigma,+}$-degree shift by one. It follows that $W^{\mathsf{prim},B}_{L_m} \simeq W^{\mathsf{prim},C}_{L_m}[1]$ as $\Lambda_{L_m}^+ \times \mathbb{Z}$-graded vector spaces. Applying Theorem \ref{thm:sdChowInterpretation} completes the proof.
\end{proof}

Variations of Theorem \ref{thm:orthSymDuality} arise by choosing different duality structures on $L_m$. Fix $1 \leq m_0 \leq m$ and consider the duality structures
\[
(s; \tau_+, \tau_-) = (-1; m-m_0, m_0), \qquad 
(s^{\prime}; \tau^{\prime}_+, \tau^{\prime}_-) = (1; m_0-1,m-m_0 +1).
\]
The corresponding varieties of self-dual representations are
\[
R_e^{\sigma} = \mathfrak{sp}_{2e}(\mathbb{C})^{\oplus m_0} \oplus (\bigwedge\nolimits^2 \mathbb{C}^{2e})^{\oplus m - m_0}, \qquad \mathsf{G}_e^{\sigma} = \mathsf{Sp}_{2e}(\mathbb{C})
\]
and, in even dimensions,
\[
R_e^{\prime \sigma} =\mathfrak{so}_{2e}(\mathbb{C})^{\oplus m - m_0 +1} \oplus (\Sym^2 \mathbb{C}^{2e})^{\oplus m_0-1}, \qquad \mathsf{G}_e^{\prime \sigma} = \mathsf{O}_{2e}(\mathbb{C}).
\]
The vector space $\mathbb{C}^{2e}$ denotes the fundamental representation of $\mathsf{Sp}_{2e}(\mathbb{C})$ or $\mathsf{O}_{2e}(\mathbb{C})$, as appropriate. Up to a twist by an automorphism of $\mathcal{H}_{L_m}$, the associated Hall modules $\mathcal{M}_{L_m}^C$ and $\mathcal{M}_{L_m}^{\prime D}$ are isomorphic. Arguing as in the proof of Theorem \ref{thm:orthSymDuality} we see that the rational Chow groups of the associated moduli spaces of $\sigma$-stable self-dual representations are isomorphic. Similarly, for $m \geq 1$ and duality structures
\[
(s; \tau_+, \tau_-) = (1; 1, m-1), \qquad
(s^{\prime}; \tau^{\prime}_+, \tau^{\prime}_-) = (1; m, 0)
\]
we obtain a twisted isomorphism between  $\mathcal{M}_{L_m}^B$ and $\mathcal{M}_{L_m}^{ \prime D}[1]$ and hence an isomorphism of Chow groups.

Finally, we explain how to compute $\Omega_{L_m}^{\sigma}$ for any duality structure. After fixing a Lie type $B$, $C$ or $D$, in \cite[Theorem 4.6]{mbyoung2016b} it is proved that $\mathcal{M}_{L_m}$ is freely generated by $W^{\mathsf{prim}}_{L_m}$ as a module over an explicitly defined subalgebra $\widetilde{\mathcal{H}}_{L_m} \subset \mathcal{H}_{L_m}$. At the level of generating series this implies the factorization
\[
A^{\sigma}_{L_m} = \widetilde{A}_{L_m} \Omega^{\sigma}_{L_m}
\]
Here we have written the parity twisted Hilbert-Poincar\'{e} series of $\mathcal{M}_{L_m}$ as
\[
A_{L_m}^{\sigma} =  \sum_{e \in \Lambda_{L_m}^{\sigma, +}} \frac{(-q^{\frac{1}{2}})^{\mathcal{E}(e)}}{\prod_{j=1}^{\lfloor \frac{e}{2} \rfloor} (1-q^{2j})} \xi^e
\]
with $e$ restricted to be even or odd depending on the type. The parity twisted Hilbert-Poincar\'{e} series $\widetilde{A}_{L_m}$ of $\widetilde{\mathcal{H}}_{L_m}$ can be written in terms of the $q$-Pochhammer symbol $(x; q)_{\infty} = \prod_{k=0}^{\infty}(1-xq^k)$ and the $\mathbb{Z}_2$-equivariant DT invariants:
\[
\widetilde{A}_{L_m} = \prod_{\substack{(e,k) \in \Lambda_{L_m}^{\sigma,+} \times \mathbb{Z} \\ \lambda \in \{ \pm \}}} (q^{\frac{k}{2} + \delta_{-1,\lambda}} \xi^e ; q^2)_{\infty}^{- \widetilde{\Omega}^{\lambda}_{L_m,(e,k)}}.
\]
Explicitly, writing the motivic DT invariant of $L_m$ as
\[
\Omega_{L_m}(q^{\frac{1}{2}}, t) = \sum_{(d,k) \in \Lambda_{L_m}^+ \times \mathbb{Z}}  \Omega_{L_m,(d,k)}q^{\frac{k}{2}} t^d \in \mathbb{Z}[q^{\frac{1}{2}}, q^{-\frac{1}{2}}] \pser{ t },
\]
we have $\widetilde{\Omega}_{L_m,(2d+1,k)}^{\pm} =0$ and 
\[
\widetilde{\Omega}_{L_m,(2d,k)}^+ = \begin{cases} \Omega_{L_m,(d,k)} & \mbox{if } \chi(e,d) + \mathcal{E}(d) + \frac{k -\chi(d,d)}{2} \equiv 0 \mod 2, \\
0 & \mbox{if } \chi(e,d) + \mathcal{E}(d) + \frac{k -\chi(d,d)}{2} \equiv 1 \mod 2
\end{cases}
\]
and
\[
\widetilde{\Omega}_{L_m,(2d,k)}^- = \begin{cases} 0 & \mbox{if } \chi(e,d) + \mathcal{E}(d) + \frac{k -\chi(d,d)}{2} \equiv 0 \mod 2, \\
\Omega_{L_m, (d,k)} & \mbox{if }  \chi(e,d) + \mathcal{E}(d) + \frac{k -\chi(d,d)}{2} \equiv 1 \mod 2. 
\end{cases}
\]

\begin{Ex}
Suppose that $m=3$. Using \cite[Theorem 6.8]{reineke2012} we find that the motivic DT invariant is given by
\begin{multline*}
\Omega_{L_3} = q^{-1} t + q^{-4} t^2 + q^{-9}(1 + q^2 + q^3) t^3 + \\
q^{-16}(1 + q^2 + q^3 + 2 q^4 + q^5 + 2q^6 +q^7 + q^8) t^4 + O(t^5).
\end{multline*}
Take the duality structure $(s; \tau_+,\tau_-)=(1;0,m)$. Then the motivic orientifold DT invariant is
\begin{multline*}
\Omega_{L_3}^B = \xi + q^{-3} \xi^3 + q^{-10}(1 +q^2 + 2 q^4) \xi^5 + \\
q^{-21}(1+ q^2 + 2 q^4 + 3 q^6 + 4 q^8 + 4 q^{10} + 4 q^{12} + q^{14})\xi^7 + \\ q^{-36}(1 + q^2 + 2q^4 + 3 q^6 + 5 q^8 + 6 q^{10} + 9 q^{12}  + 10q^{14} +  \\  13 q^{16} + 14 q^{18} + 15 q^{20} + 13 q^{22} + 10q^{24} + 3q^{26})\xi^9 + O(\xi^{11}).
\end{multline*}
Theorem \ref{thm:sdChowInterpretation} implies that the coefficient of $\xi^{2e+1}$ in $\Omega_{L_3}^B$ is the Chow theoretic Poincar\'{e} polynomial of $(\mathfrak{so}_{2e+1}^{\oplus 3})^{st} \slash \mathsf{O}_{2e+1}$, which by Theorem \ref{thm:orthSymDuality} agrees with the Chow theoretic Poincar\'{e} polynomial of $(\mathfrak{sp}_{2e}^{\oplus 3})^{st} \slash \mathsf{Sp}_{2e}$. 
\end{Ex}

\setcounter{secnumdepth}{0}

%\listoftodos

\footnotesize

\bibliographystyle{plain}
\bibliography{mybib}

\def\cprime{$'$}
\begin{thebibliography}{10}

\bibitem{atiyah1983}
M.~Atiyah and R.~Bott.
\newblock The {Y}ang-{M}ills equations over {R}iemann surfaces.
\newblock {\em Philos. Trans. Roy. Soc. London Ser. A}, 308(1505):523--615,
  1983.

\bibitem{brion1998}
M.~Brion.
\newblock Equivariant cohomology and equivariant intersection theory.
\newblock In {\em Representation theories and algebraic geometry ({M}ontreal,
  {PQ}, 1997)}, volume 514 of {\em NATO Adv. Sci. Inst. Ser. C Math. Phys.
  Sci.}, pages 1--37. Kluwer Acad. Publ., Dordrecht, 1998.
\newblock Notes by Alvaro Rittatore.

\bibitem{brion2012b}
M.~Brion and R.~Joshua.
\newblock Notions of purity and the cohomology of quiver moduli.
\newblock {\em Internat. J. Math.}, 23(9):1250097, 30, 2012.

\bibitem{brown1982}
E.~Brown, Jr.
\newblock The cohomology of {$B{\rm SO}_{n}$} and {$B{\rm O}_{n}$} with integer
  coefficients.
\newblock {\em Proc. Amer. Math. Soc.}, 85(2):283--288, 1982.

\bibitem{chen2014}
Z.~Chen.
\newblock Geometric construction of generators of {C}o{HA} of doubled quiver.
\newblock {\em C. R. Math. Acad. Sci. Paris}, 352(12):1039--1044, 2014.

\bibitem{davison2016a}
B.~Davison and S.~Meinhardt.
\newblock Cohomological {D}onaldson-{T}homas theory of a quiver with potential
  and quantum enveloping algebras.
\newblock ar{X}iv:1601.02479, 2016.

\bibitem{denef2010}
F.~Denef, M.~Esole, and M.~Padi.
\newblock Orientiholes.
\newblock {\em J. High Energy Phys.}, (3):045, 44, 2010.

\bibitem{derksen2002}
H.~Derksen and J.~Weyman.
\newblock Generalized quivers associated to reductive groups.
\newblock {\em Colloq. Math.}, 94(2):151--173, 2002.

\bibitem{dhillonyoung2016}
A.~Dhillon and M.~Young.
\newblock The motive of the classifying stack of the orthogonal group.
\newblock {\em Michigan Math. J.}, 65(1):189--197, 2016.

\bibitem{edidin1997}
D.~Edidin and W.~Graham.
\newblock Characteristic classes in the {C}how ring.
\newblock {\em J. Algebraic Geom.}, 6(3):431--443, 1997.

\bibitem{edidin1998}
D.~Edidin and W.~Graham.
\newblock Equivariant intersection theory.
\newblock {\em Invent. Math.}, 131(3):595--634, 1998.

\bibitem{efimov2012}
A.~Efimov.
\newblock Cohomological {H}all algebra of a symmetric quiver.
\newblock {\em Compos. Math.}, 148:1133--1146, 2012.

\bibitem{franzen2015}
H.~Franzen.
\newblock On the semi-stable {C}o{H}a and its modules arising from smooth
  models.
\newblock ar{X}iv:1502.04327, 2015.

\bibitem{franzen2016}
H.~Franzen.
\newblock On cohomology rings of non-commutative {H}ilbert schemes and
  {C}o{H}a-modules.
\newblock {\em Math. Res. Lett.}, 23(3):804--840, 2016.

\bibitem{franzen2015b}
H.~Franzen and M.~Reineke.
\newblock Semi-stable {C}how-{H}all algebras of quivers and quantized
  {D}onaldson-{T}homas invariants.
\newblock ar{X}iv:1512.03748, 2015.

\bibitem{harada2011}
M.~Harada and G.~Wilkin.
\newblock Morse theory of the moment map for representations of quivers.
\newblock {\em Geom. Dedicata}, 150:307--353, 2011.

\bibitem{hausel2013}
T.~Hausel, E.~Letellier, and F.~Rodriguez-Villegas.
\newblock Positivity for {K}ac polynomials and {DT}-invariants of quivers.
\newblock {\em Ann. of Math. (2)}, 177(3):1147--1168, 2013.

\bibitem{king1994}
A.~King.
\newblock Moduli of representations of finite-dimensional algebras.
\newblock {\em Quart. J. Math. Oxford Ser. (2)}, 45(180):515--530, 1994.

\bibitem{kirwan1984}
F.~Kirwan.
\newblock {\em Cohomology of quotients in symplectic and algebraic geometry},
  volume~31 of {\em Mathematical Notes}.
\newblock Princeton University Press, Princeton, NJ, 1984.

\bibitem{kontsevich2011}
M.~Kontsevich and Y.~Soibelman.
\newblock Cohomological {H}all algebra, exponential {H}odge structures and
  motivic {D}onaldson-{T}homas invariants.
\newblock {\em Commun. Number Theory Phys.}, 5(2):231--352, 2011.

\bibitem{laumon1996}
G.~Laumon and M.~Rapoport.
\newblock The {L}anglands lemma and the {B}etti numbers of stacks of
  {$G$}-bundles on a curve.
\newblock {\em Internat. J. Math.}, 7(1):29--45, 1996.

\bibitem{letellier2015}
E.~Letellier.
\newblock D{T}-invariants of quivers and the {S}teinberg character of {${\rm
  GL}_n$}.
\newblock {\em Int. Math. Res. Not. IMRN}, (22):11887--11908, 2015.

\bibitem{meinhardt2014}
S.~Meinhardt and M.~Reineke.
\newblock Donaldson-{T}homas invariants versus intersection cohomology of
  quiver moduli.
\newblock ar{X}iv:1411.4062, 2014.

\bibitem{pandharipande1998}
R.~Pandharipande.
\newblock Equivariant {C}how rings of {${\rm O}(k),\ {\rm SO}(2k+1)$}, and
  {${\rm SO}(4)$}.
\newblock {\em J. Reine Angew. Math.}, 496:131--148, 1998.

\bibitem{reineke2003}
M.~Reineke.
\newblock The {H}arder-{N}arasimhan system in quantum groups and cohomology of
  quiver moduli.
\newblock {\em Invent. Math.}, 152(2):349--368, 2003.

\bibitem{reineke2012}
M.~Reineke.
\newblock Degenerate cohomological {H}all algebra and quantized
  {D}onaldson-{T}homas invariants for {$m$}-loop quivers.
\newblock {\em Doc. Math.}, 17:1--22, 2012.

\bibitem{totaro1999}
B.~Totaro.
\newblock The {C}how ring of a classifying space.
\newblock In {\em Algebraic {$K$}-theory ({S}eattle, {WA}, 1997)}, volume~67 of
  {\em Proc. Sympos. Pure Math.}, pages 249--281. Amer. Math. Soc., Providence,
  RI, 1999.

\bibitem{mbyoung2015}
M.~Young.
\newblock Self-dual quiver moduli and orientifold {D}onaldson-{T}homas
  invariants.
\newblock {\em Commun. Number Theory Phys.}, 9(3):437--475, 2015.

\bibitem{mbyoung2016b}
M.~Young.
\newblock Representations of cohomological {H}all algebras and
  {D}onaldson-{T}homas theory with classical structure groups.
\newblock ar{X}iv:1603.05401, 2016.

\end{thebibliography}
 
\end{document}